\documentclass[11pt,a4paper]{amsart}

\usepackage{amsmath,amssymb,amsthm, amsfonts,enumerate,mathtools,pinlabel,float,stmaryrd,makecell}
\usepackage[all]{xy}
\usepackage[mathscr]{eucal}
\usepackage{amsbsy}
\usepackage{mathtools}
\usepackage{pinlabel}
\usepackage{graphicx}
\usepackage{multirow}
\usepackage{tabularx}
\usepackage{lscape}
\usepackage[myheadings]{fullpage}
\usepackage{rotating}
\usepackage{tikz-cd,tikz,pgf}
\usepackage{tikz-cd}
\usetikzlibrary{arrows.meta, bending}
\tikzset{C/.style={circle, minimum size=8mm,
                   node contents={},
                   append after command={\pgfextra{%
        \draw[-{Straight Barb[flex']}](\tikzlastnode.150) arc (150:450:4mm);}
                }}
        }
\usepackage{caption}
\usepackage{subcaption}
\usepackage{cite}
\usepackage{array}
\usepackage{xcolor}
\usepackage{calc}
\usepackage{comment}
\usepackage{soul}
\usepackage{hyperref}
\hypersetup{
    colorlinks=true,   
    linkcolor=blue,
    citecolor=blue,
}

\newtheorem{cor*}{Corollary}
\newtheorem{prop*}{Proposition}

\newtheorem{theorem}{Theorem}[section]
\newtheorem{cor}{Corollary}[theorem]
\newtheorem{prop}[theorem]{Proposition}
\newtheorem{lem}[theorem]{Lemma}

\theoremstyle{definition}
\newtheorem{defn}[theorem]{Definition}
\newtheorem{exmp}[theorem]{Example}
\newtheorem{rem}[theorem]{Remark}

\newcommand{\F}{\mathcal{F}}
\newcommand{\G}{\mathcal{G}}
\newcommand{\Z}{\mathbb{Z}}
\newcommand{\N}{\mathbb{N}}
\newcommand{\Orb}{\mathcal{O}}

\newcommand{\D}{\mathscr{D}}
\newcommand{\wlp}{(\mathscr{D}_a,(D_{\overline{G}},\Pi_{\overline{G}}))}

\DeclareMathOperator{\sig}{\mathrm{sig}}
\DeclareMathOperator{\lcm}{lcm}
\DeclareMathOperator{\map}{\mathrm{Mod}}

\DeclareMathOperator{\homeo}{\mathrm{Homeo}^{+}}

\everymath{\displaystyle}

\makeatletter
\@namedef{subjclassname@2022}{\textup{2022} Mathematics Subject Classification}
\makeatother

\begin{document}

\title[Alternating and symmetric actions on surfaces]{Alternating and symmetric \\ actions on surfaces}

\author{Rajesh Dey}
\address{Department of Mathematics\\
Indian Institute of Science Education and Research Bhopal\\
Bhopal Bypass Road, Bhauri \\
Bhopal 462 066, Madhya Pradesh\\
India}
\email{rajesh17@iiserb.ac.in}

\author{Kashyap Rajeevsarathy}
\address{Department of Mathematics\\
Indian Institute of Science Education and Research Bhopal\\
Bhopal Bypass Road, Bhauri \\
Bhopal 462 066, Madhya Pradesh\\
India}
\email{kashyap@iiserb.ac.in}
\urladdr{https://home.iiserb.ac.in/$_{\widetilde{\phantom{n}}}$kashyap/}



\begin{abstract}
Let $\mathrm{Mod}(S_g)$ be the mapping class group of the closed orientable surface of genus $g \geq 2$. In this article, we derive necessary and sufficient conditions under which two torsion elements in $\mathrm{Mod}(S_g)$ will have conjugates that generate a finite symmetric or an alternating subgroup of $\mathrm{Mod}(S_g)$. Furthermore, we characterize when an involution would lift under the branched cover induced by an alternating action on $S_g$. Moreover, up to conjugacy, we derive conditions under which a given periodic mapping class is contained in a symmetric or an alternating subgroup of $\mathrm{Mod}(S_g)$. In particular, we show that symmetric or alternating subgroups can not contain irreducible mapping classes and hyperelliptic involutions. Finally, we classify the symmetric and alternating actions on $S_{10}$ and $S_{11}$ up to a certain equivalence we call weak conjugacy.
\end{abstract}

\maketitle

\section{Introduction} 
\label{sec:intro}
Let $S_g$ be the closed, connected, and orientable surface of genus $g$. Let $\homeo(S_g)$ be the group of orientation-preserving homeomorphisms of $S_g$, and let $\map(S_g) := \pi_0(\homeo(S_g))$ be the mapping class group of $S_g$. Throughout this article, we will assume that $g\geq2$ unless specified otherwise. Given $F, G \in \map(S_g)$ of finite order, a pair of conjugates $F'$, $G'$ 
(of $F, G$ resp.) may or may not generate a subgroup isomorphic to $\langle F, G \rangle$. For example, consider the periodic mapping classes $F, G \in \map(S_5)$ represented by homeomorphisms $\F, \G \in \homeo(S_5)$ where $\F$ is a $\pi/2$ and $\G$ is a $\pi$ rotation of $S_5$ embedded in $\mathbb{R}^3$ in the shape of a cube (see Figure \ref{fig:cube}) with the origin $O$ as its center. Consider another pair of periodic mapping classes $F', G' \in \map(S_5)$ represented by homeomorphisms $\F, \G' \in \homeo(S_5)$ where $\mathcal{G'}$ is the $\pi$ rotation as shown in Figure \ref{fig:cube} below.
\begin{figure}[ht]
\labellist
\tiny
\pinlabel $O$ at 303 488
\pinlabel $X$ at 390 495
\pinlabel $\F$ at 500 500
\pinlabel $\pi/2$ at 480 560
\pinlabel $Y$ at 245 416
\pinlabel $\G$ at 234 375
\pinlabel $\pi$ at 165 391
\pinlabel $Z$ at 272 587
\pinlabel $\G'$ at 241 665
\pinlabel $\pi$ at 200 612
\endlabellist
\centering
 \includegraphics[scale=.4]{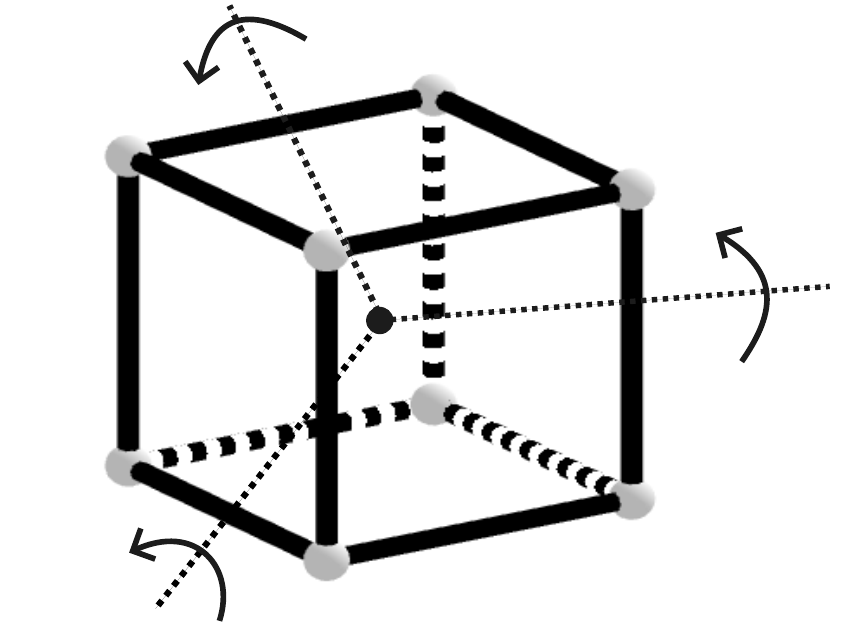}
  \caption{An $\Sigma_4$-action on $S_5$.}
  \label{fig:cube}
\end{figure}
\noindent Since the group of symmetries of a cube is $\Sigma_4$, it is apparent that $\langle \F,\G\rangle\cong \langle (1~2~3~4),(1~2)\rangle\cong \Sigma_4$ and $\langle \F,\G'\rangle\cong \langle (1~2~3~4),(1~3)\rangle\cong D_8$ but $\G'$ is conjugate to $\G$. As any finite subgroup of $\homeo(S_g)$ embeds into $\map(S_g)$ under the natural projection map, it follows that $\langle F,G \rangle\cong\Sigma_4$ and $\langle F',G'\rangle\cong D_8$. This example leads to the following natural question: Given $F, G \in \map(S_g)$ of finite orders, can one derive necessary and sufficient combinatorial conditions under which there exist conjugates $F'$, $G'$ (of $F, G$ resp.) such that $\langle F', G'\rangle$ is isomorphic to a chosen two-generator finite group. This question has already been answered affirmatively for two-generator abelian groups in \cite{Dhanwani} and metacyclic groups in \cite{apeksha,apeksha2}.

The main results of this paper (in Section~\ref{sec:main}) provide an affirmative answer to this question for alternating (see Theorem \ref{thm:main1}) and symmetric groups (see Theorem \ref{thm:main2}). By the Nielsen-Kerckhoff theorem~\cite{nielsen,kerckhoff}, the conjugacy classes of periodic maps in $\homeo(S_g)$ are in one-to-one correspondence with the conjugacy classes of periodic mapping classes in $\map(S_g)$ via the natural projection map. Thus, a key ingredient in the proof of the main results (see Proposition \ref{prop:main1} and Proposition \ref{prop:main2}) is a combinatorial characterization of alternating (and symmetric) actions on $S_g$ up to certain equivalence called \textit{weak conjugacy} (see Definition \ref{defn:wk_conjugacy}) that arises naturally in our context. We use the theory of finite group actions on surfaces \cite{macbeath,broughton}, Thurston's orbifold theory \cite[Chapter 16]{Thurston}, and Ore's Theorem~\cite{Ore} to prove our main result.

The problem of liftability of the mapping classes originated from the pioneering works of Birman-Hilden in the early 1970s that started with~\cite{birman1971} and concluded with \cite{birman} (see also~\cite{margalit-winarski}). Since then, though the liftable mapping class groups under finite abelian covers have been widely studied \cite{broughton_normalizer,ghaswala_superelliptic,ghaswala_sphere,pankaj}, such an analysis is yet to be undertaken for covers induced by finite simple groups. This motivates our pursuit in Section \ref{section:involutions}, where we use our main results and an algebraic analog of liftability criterion, due to Broughton (see  \cite[Theorem 3.2.]{broughton_normalizer}), to characterize $A_n\rtimes \Z_2$-action as an extension of $A_n$-action on $S_g$.  To this end, we derive necessary and sufficient conditions for the conjugacy class of a $\Z_2$-action to lift under an alternating branched cover to the weak conjugacy class of a $\Sigma_n$-action on $S_g$ (see Theorem \ref{thm:wls}). We also provide a sufficient condition for this $\Z_2$-action to lift to a $A_n\times \Z_2$-action on $S_g$ (see Theorem \ref{thm:alternating_extensions}). As a consequence of these results, we establish the existence of several alternating subgroups of $\map(S_g)$ that are self-normalizing (see Corollary~\ref{cor:self_normalizing}). In particular, we have the following. 
\begin{cor*}
An $A_n$-action on $S_g$ whose quotient orbifold is a sphere with three cone points of distinct orders does not extend either to a $\Sigma_n$-action or an $A_n \times \Z_2$-action.
\end{cor*}

\noindent Another application is the following result, which pertains to the extensions of free $A_n$-actions. 
\begin{cor*}
There exists a free $A_n$-action on $S_g$ if and only if there exists $k\in \mathbb{N}$ such that $g= 1+k\cdot\frac{n!}{2}$. Furthermore:
\begin{enumerate}[(i)]
\item A free $A_n$-action on $S_g$ extends to a free $\Sigma_n$-action on $S_g$ if and only if $k$ is even, and 
\item A free $A_n$-action on $S_g$ extends to a non-free $\Sigma_n$-action on $S_g$ if $k \geq 3$ and $k$ is odd.
\end{enumerate}
\end{cor*}

In Section \ref{section:embedding}, we derive a few other applications of our main theorems. The following result provides an asymptotic bound on the order of an element in an alternating or symmetric subgroup of $\map(S_g)$. 
\begin{prop*}
Let $H\leq \mathrm{Mod}(S_g)$ such that $H\cong \Sigma_n \text{(or } A_n)$. Then, for sufficiently large $n$, any element in $H$ has order less than $g/k$, for some $k\in \mathbb{N}$.
\end{prop*}
\noindent An \textit{irreducible periodic mapping class} is characterized by its corresponding quotient orbifold, which is a sphere with three cone points (see~\cite{gilman}). As an application of the proposition above and our main theorems, we have the following corollary. 
\begin{cor*}
Let $H\leq \mathrm{Mod}(S_g)$ such that $H\cong \Sigma_n$ or $A_n$. Then $H$ can not contain an irreducible mapping class or a hyperelliptic involution.
\end{cor*}

\noindent Finally, in Section \ref{section:table}, we have classified the alternating and symmetric subgroups of $\map(S_{10})$ and $\map(S_{11})$ up to \textit{weak conjugacy}.
\section{Preliminaries}
In this section, we introduce some basic concepts and results from the theory of group actions on surfaces that will be extensively used in this paper.
\subsection{Fuchsian groups} Let $\mathbb{H}$ be the upper-half plane equipped with the standard hyperbolic metric. A \textit{cocompact Fuchsian group}~\cite{SK1} $\Gamma$ is a discrete subgroup of $\mathrm{PSL}(2,\mathbb{R})$, the group of orientation-preserving isometries of $\mathbb{H}$, such that $\mathbb{H}/\Gamma$ is compact. A cocompact Fuchsian group $\Gamma$ admits a presentation given by:
\[\Gamma=\biggl\langle~ \alpha_1,\beta_1,..., \alpha_{g_{0}},\beta_{g_{0}},\xi_{1},...,\xi_{r}~\bigg|~\xi_{1}^{m_1},...,\xi_{r}^{m_r},\prod_{i = 1}^{r} \xi_{i}\prod_{i = 1}^{g_{0}} [\alpha_{i},\beta_{i}]~\biggr\rangle. \tag{1} \label{eqn:eqn1}
\]
The \textit{signature} of $\Gamma$ is defined to be the tuple, 
$$\sig(\Gamma):=(g_0; m_1, \dots, m_r).$$
The quotient space $\mathcal{O} = \mathbb{H}/\Gamma$ is called an \textit{hyperbolic orbifold of genus $g_0$} with \textit{orbifold signature} $\sig(\mathcal{O}) := \sig(\Gamma$). The orbifold $\mathcal{O}$ has $r$ distinguished points $x_1,\ldots,x_r$ called \textit{cone points}, of orders $m_1,\ldots m_r$, respectively, whose local neighborhoods are isometric to the quotient of an open ball in $\mathbb{H}$ by a finite group of rotations of $\mathbb{H}$. In particular, if $\Gamma$ has no torsion, then everything matches our usual notions of hyperbolic surfaces.

\subsection{Finite group actions on surfaces} By the Nielsen realization theorem~\cite{nielsen,kerckhoff}, a finite subgroup $H < \map(S_g)$ can be realized as a group of isometries under a hyperbolic metric $\rho$ on $S_g$. By Thurston's orbifold theory~\cite[Chapter 16]{Thurston}, the action of $H$ on $S_g$ induces a short exact sequence $$1 \longrightarrow \pi_1(S_g) \longrightarrow \pi _1^{orb}(\mathcal{O}_H)\overset{\phi_H} \longrightarrow H \longrightarrow 1,$$
where $\mathcal{O}_H$ is the orbifold $S_g/H(\approx \mathbb{H}/\pi _1^{orb}(\mathcal{O}_H))$, and we have the following version of the Riemann's existence theorem~\cite[Corollary 5.9.5]{Jones} due to Harvey~\cite{harvey}. 
\begin{theorem}
\label{thm:riemann}
A finite group $H$ acts on $S_g$ with a quotient orbifold $\Orb_H$ with $\sig(\Orb_H) = (g_0; m_1,\dots, m_r)$ if and only if the following conditions hold.
\begin{enumerate}[(i)]

    \item There exists a co-compact Fuchsian group $\Gamma$ with $\sig(\Gamma) = \sig(\Orb_H)$.
    \item There exists a surjective homomorphism
$$\phi_H : \Gamma \to H,$$ which is order-preserving on torsion elements.
    \item $\frac{2-2g}{|H|}= 2-2g_0-\sum_{i=1}^{r} (1-\frac{1}{m_i}).$
\end{enumerate}
\end{theorem}

\noindent The group $\Gamma$ in Theorem~\ref{thm:riemann} is also known as its associated \textit{orbifold fundamental group} denoted by $\pi _1^{orb}(\mathcal{O}_H)$.  The map $\phi_H$ is called the \textit{surface kernel epimorphism} associated with the action of $H$ on $S_g$. We note that the epimorphism $\phi_H$ in Theorem~\ref{thm:riemann} uniquely determines a faithful action of $H$ on $S_g$, which we denote by $\epsilon_H : H \rightarrow \homeo(S_g)$. We will also require the following result from the theory of Riemann surfaces~\cite[Lemma 11.5]{breuer}. 

\begin{lem}
\label{lem:subaction}
Let $H<\homeo(S_g)$ be of finite order with $\sig(\mathcal{O}_H)=(g_0;m_1,\ldots,m_l)$, and let $\G\in H$ be of order $m$. Then for $u\in \Z_m^{\times}$, we have
$$|\mathbb{F}_{\G}(u,m)|=|C_H(\G)|\cdot\sum\limits_{\substack{1\leq i\leq l, ~m\mid m_i \\ \G \sim_H \phi_H(\xi_i)^{m_iu/m}
}}\frac{1}{m_i},  $$
where $\mathbb{F}_{\G}(u,m)$ denote the set of fixed points of $\G$ with induced rotation angle $2\pi u^{-1}/m$ and $C_H(\G)$ denotes the centralizer of $\G$ in $H$.
\end{lem}

\subsection{Cyclic actions on surfaces} As we saw earlier, an $F\in \map(S_g)$ of order $n$ is represented by a hyperbolic isometry in $\homeo(S_g)$ called the \textit{Nielsen representative}, which we denote by $\F$. The orbifold $S_g/\langle \F \rangle$ has $l$ cone points $x_1,\ldots,x_l$ of orders $m_1,\ldots, m_l$. Each $x_i$ lifts under the branched cover $S_g \to S_g/\langle \F \rangle$ to an orbit of size $n/m_i$ on $S_g$ and the points in this orbit have an induced local rotation given by $2\pi c_i^{-1}/m_i$, where $\gcd(c_i,m_i)=1$. Thus, we can associate the following combinatorial data with $\langle \F \rangle$-action on $S_g$.
\begin{defn}\label{defn:cyclic_data_set}
By a \textit{cyclic data set}, we mean a tuple
$$D=(n, g_0; (c_1, m_1),\dots , (c_l, m_l)),$$ where $n\geq2$ and $g_0\geq0$ are integers such that:
\begin{enumerate}[(i)]
\item $m_i \geq 2$, $m_i \vert n $ and $\gcd(c_i,m_i)=1 \text{ for all } i, $
\item $\lcm(m_1, \dots , \hat{m_i}, \dots, m_l)$ = $\lcm(m_1, \dots, m_l) \text{ for all } i$, and if $g_0 = 0$, then $\lcm(m_1, \dots, m_l) = n$,
\item $\sum_{i=1}^{l} \frac{n}{m_i}c_i\equiv 0 \pmod{n}$.
\end{enumerate}
We call $n$ as the \textit{degree} of the data set, and the positive integer,
$g$ determined by
    $$\frac{2-2g}{n}=2-2g_{0}-\sum_{i=1}^{l} \left(1-\frac{1}{m_i}\right),$$
as the \textit{genus} of the data set. 
\end{defn}

\noindent If a pair $(c_i,m_i)$ in $D$ occurs more than once, we will use the symbol $(c_i,m_i)^{[\ell_i]}$ to denote that the pair $(c_i,m_i)$ occurs with multiplicity $\ell_i$ in $D$. As a result of Nielsen~\cite{JN1}, this data is known to characterize cyclic actions up to conjugacy. 
\begin{theorem}
The cyclic data sets of genus $g$ and degree $n$ correspond to conjugacy classes
of periodic maps of order $n$ on $S_g$.
\end{theorem}
\noindent By the Nielsen realization theorem \cite{nielsen,kerckhoff}, cyclic data sets of genus $g$ and degree $n$ are also in one-to-one correspondence with the conjugacy classes of order-$n$ periodic mapping classes in $\map(S_g)$. Therefore, we represent the conjugacy class of a periodic $F\in\map(S_g)$ by $D_{F}$. 
\subsection{Symmetric and alternating actions on surfaces} 
\label{subsec:sym_alt_actions}
We will denote symmetric and alternating group on the $n$ symbols $\{1,2,\ldots,n\}$ by $\Sigma_n$ and $A_n$, respectively.  Given $\sigma_1,\sigma_2 \in \Sigma_n$, we will fix the convention that the product $\sigma_1\sigma_2$ is defined as the composition $\sigma_1 \circ \sigma_2$ of the permutations when treated as functions. We denote the cycle type of each permutation $\sigma \in \Sigma_n$ by the tuple $(k_1, \ldots k_{\ell})$, where $k_i \leq k_{i+1}$ and each $k_i>1$ for $1\leq i \leq \ell$. It is well known that the conjugacy class for any $\sigma \in \Sigma_n$ is completely determined by its cycle type. The same holds true for most $ \sigma \in A_n$, except when $\sigma$ decomposes (canonically) into disjoint cycles of distinct odd lengths, where fixed points are treated as cycles of length $1$, in which case the conjugacy class of $\sigma$ further splits into two subclasses of equal size. We will fix the standard generators $\Sigma_n = \langle \sigma_s, \tau_s \rangle$ and $A_n = \langle \sigma_a, \tau_a\rangle$, where $\sigma_s= (1~2)$, $\tau_s=(1~2~\dots~n)$, $\sigma_a= (1~2~3)$, and 
$$\tau_a = 
 \begin{cases}
    (1~2~\dots~n), & \text{if } n \text{ is odd, and} \\
    (2~3~\dots~n), &  \text{if } n \text{ is even.}
  \end{cases}$$
Our main theorems use the following basic facts about $\Sigma_n$ and $A_n$.
\begin{theorem}
For $n\neq 6$, every automorphism of $\Sigma_n$ is inner. Furthermore, if $n\geq3$ and $n\neq6$, then every automorphism of $A_n$ is the restriction of an inner automorphism of $\Sigma_n$.
\label{lem:inner}
\end{theorem}
\noindent We will also require the following classical result due to Ore \cite{Ore} in the proof of our main theorems.
\begin{theorem}
For $n\geq 5$, every element in $A_n$ is a commutator.
\label{thm:Ore}
\end{theorem}
\noindent We will require the following basic fact from group theory.
\begin{lem} Let $H$ be a group and $c_1, c_2, \dots, c_r \in H$. For each permutation $\pi$ of $\{1,2,\dots,r\}$ there are elements $h_1, h_2,\dots, h_r \in H$ such that $\prod_{i = 1}^{r} h_ic_{\pi(i)}h_{i}^{-1}$ is $H$-conjugate to $\prod_{i = 1}^{r} c_i$.
\label{lem:perm}
\end{lem}
\noindent Let $\Gamma$ be a co-compact Fuchsian group having the fixed presentations as in Equation \ref{eqn:eqn1}. Following the notation in Equation \ref{eqn:eqn1}, we will state a result of Zieschang \cite[Theorem 5.8.2]{zieschang}, which paraphrases to the following since we only deal with orientation preserving map (see also \cite{broughton_normalizer}).
\begin{theorem}
\label{thm:zieschang}
Let $\psi:\Gamma \to \Gamma$ be an automorphism. Then $\psi(\xi_i)$ is conjugate to $\xi_{\pi(i)}$ for some permutation $\pi$ on $r$ letters.

\end{theorem}
\noindent we will use the above result in the proof of our main theorems.
\begin{defn}
Let $H< \map(S_g)$ be a finite subgroup. Then periodic mapping classes $F, G \in \map(S_g)$ are said to \textit{weakly generate} $H$ (in symbols $\langle F, G\rangle_w=H$) if there exists conjugates $F', G'$ of $F, G$ respectively in $\map(S_g)$ such that $\langle F', G'\rangle=H$.
\end{defn}
\noindent We now define an equivalence of actions that will play a crucial role in the theory we will develop.
\begin{defn}\label{defn:wk_conjugacy}
Consider $H_1,H_2 < \homeo(S_g)$ such that $H_1 \cong H_2 \cong A_n$ (or $\Sigma_n$). Then we say that the actions of $H_1$ and $H_2$ on $S_g$ are \textit{weakly conjugate} if there exists an isomorphism $\psi: \pi_1^{orb}(\Orb_{H_1}) \to\pi_1^{orb}(\Orb_{H_2})$ and an isomorphism $\chi : H_1 \to H_2$ such that $(\chi\circ\phi_{H_1})(g)$ is conjugate to $(\phi_{H_2}\circ\psi)(g)$ in $H_2$, whenever $g\in\pi_1^{orb}(\Orb_{H_1})$ is of finite order.
\end{defn}
\noindent In other words, the action of $H_1$ and $H_2$ on $S_g$ are weakly conjugate if the following diagram commutes up to conjugacy at the level of torsion elements.
\begin{center}\label{fig:wkconjugacy}
\begin{tikzcd}[sep=12mm,
arrow style=tikz,
arrows=semithick,
diagrams={>={Straight Barb}}
                ]
\pi_1^{orb}(\Orb_{H_1}) \dar["\phi_{H_1}" '] \rar["\psi",""name=U]
        & \pi_1^{orb}(\Orb_{H_2}) \dar["\phi_{H_2}"]      \\
H_1\rar["\chi"  ,""name=D] & H_2
                     \ar[to path={(U) node[pos=.5,C] (D)}]{}
    \end{tikzcd}
\captionof{figure}{The weak conjugacy of the $H_i$-actions on $S_g$.}
\end{center}
The notion of weak conjugacy defines an equivalence relation on alternating (or symmetric) actions on $S_g$, and the equivalence classes thus obtained will be called \textit{weak conjugacy classes.} We represent the weak conjugacy class of an $H$-action on $S_g$ by $[\phi_H]$. We conclude this section with the following remark.
\begin{rem}
\label{rem:weak_conj}
By virtue of the Nielsen realization theorem, the notion of a weak conjugacy class of an $H$-action on $S_g$, where $H \cong A_n$ (or $\Sigma_n$) induces an analogous notion of weak conjugacy class of a subgroup of $\map(S_g)$ that is isomorphic to $A_n$ (or $\Sigma_n$).
\end{rem}

\section{Main Theorems}
\label{sec:main}
In this section, we will derive a combinatorial characterization of the weak conjugacy classes of $A_n$ and $\Sigma_n$-actions on $S_g$. 
\subsection{Alternating actions} We begin with the following technical definition.
\begin{defn}
\label{defn:alt_data_set}
An \textit{alternating data set $\D_a$ of degree $n$ and genus $g\geq 2$} is an ordered tuple $$(n, g_0; [\sigma_1, m_1; k_{11},k_{21},\dots , k_{l_11}], \dots, [\sigma_r, m_r; k_{1r},k_{2r},\dots , k_{l_rr}]),$$ where $n,g_0$ are integers with $n\geq 5$, $g_0 \geq 0$, and the $\sigma_i \in A_n$, non-trivial, satisfying the following conditions:
   \begin{enumerate}[(i)]
   \item  $\displaystyle  \frac{2-2g}{n\, !/2}=  2-2g_{0}-\sum_{i=1}^{r} \left(1-\frac{1}{m_i}\right).$
     \item $m_i$ is the order of $\sigma_i$ and $(k_{1i},\ldots,k_{l_ii})$ is the cycle type of $\sigma_i$.
\item For $g_0=0$, we have:
\begin{enumerate}[(a)]
    \item $\langle \sigma_1,\ldots,\sigma_r \rangle = A_n$ and 
    \item $\prod_{i = 1}^{r} \sigma_{i}= 1$.
\end{enumerate}
\item For $g_0=1$, there exist $\sigma_{r+1}, \sigma_{r+2} \in A_n$ such that:
\begin{enumerate}[(a)]
    \item $\langle \sigma_1,\ldots,\sigma_r, \sigma_{r+1}, \sigma_{r+2} \rangle = A_n$ and
    \item $\prod_{i = 1}^{r} \sigma_{i}= [\sigma_{r+2}, \sigma_{r+1}]$. 
\end{enumerate}
        \end{enumerate}

\end{defn}
\noindent If a tuple $[\sigma_i, m_i; k_{1i},k_{2i},\dots , k_{l_ii}]$ in $\D_a$ occurs more than once, then we will use the symbol $[\sigma_i, m_i; k_{1i},k_{2i},\dots , k_{l_ii}]^{[\ell_i]}$ to denote that it occurs with multiplicity in $\ell_i$ in $\D_a$. We will now define an equivalence of alternating data sets. 
\begin{defn}
\label{defn:eq_data_sets}
Two alternating data sets $$\D_a=\biggl(n, g_0; [\sigma_i, m_i; k_{1i},\dots , k_{l_ii}]_{i=1}^{r}\biggr) ~~\text{and} ~~\D_a'=\biggl(n, g_0; [\sigma_i', m_i'; k_{1i}',\dots , k_{l_i'i}']_{i=1}^{r}\biggr)$$ are said to be \textit{equivalent} if there exists a permutation $\pi$ on $r$ letters satisfying the following conditions.
\begin{enumerate}[(i)]
    \item For $1 \leq i \leq r$, we have $l_i=l_{\pi(i)}'$ and $(k_{1i},\dots , k_{l_ii})= (k'_{1\pi(i)},k'_{2\pi(i)},\dots , k'_{l_i\pi(i)})$. 
    \item If for some $1\leq i \leq r$,  $k_{1i},\dots , k_{l_ii}$ are distinct odd integers and $\Sigma_{j=1}^{l_i}k_{ji}\geq n-1$, then exactly one of the following conditions hold.
    \begin{enumerate}[(a)]
    \item $\sigma_i$ is conjugate to $\sigma_{\pi(i)}'$ in $A_n$ for all such $i$.
    \item $\sigma_i$ is not conjugate to $\sigma_{\pi(i)}'$ in $A_n$ for all such $i$.
    \end{enumerate}
\end{enumerate}
\end{defn}

\noindent We represent the equivalence class of an alternating data set $\D_a$ by $[\D_a]$.
\begin{prop}
\label{prop:main1}
 Equivalence classes of alternating data sets of genus $g$ and degree $n$ correspond to weak conjugacy classes of $A_n$-actions on $S_g$.
\end{prop}
\begin{proof} We establish this correspondence by showing that there exists a natural map bijective map:
\begin{center}

$\left\{\parbox{57mm}{\raggedright Equivalence classes of alternating data sets of genus $g$ and degree $n$}\right\}\xrightarrow{\makebox[1.5cm]{$\Theta_a$}}
\left\{\parbox{40mm}{\raggedright weak conjugacy classes of $A_n$-actions on $S_g$} \right\}.$

\end{center}

\noindent Given an alternating data set $\D_a=\biggl(n, g_0; [\sigma_i, m_i; k_{1i},\dots , k_{l_ii}]_{i=1}^{r}\biggr)$ of genus $g$ and degree $n$, we first fix a co-compact Fuchsian group $\Gamma$ with sig$(\Gamma)=(g_0;m_1,m_2,\dots,m_r)$ having the fixed presentation as in (\ref{eqn:eqn1}). We want to construct a $H$-action on $S_g$, where $H \cong A_n$, such that $\sig(\pi_1^{orb}(\Orb_H))=(g_0; m_1, \dots, m_r).$ Thus, by Theorem \ref{thm:riemann}, it suffices to show that there exists a surface kernel epimorphism $\phi:\Gamma\to A_n$, where $\Gamma$ is taken to have the same presentation as in \ref{eqn:eqn1}. Now we define $\phi$ as follows: 
$$\begin{array}{rcll}
\phi(\xi_i) & := & \sigma_i, & \text{for } g_0 \geq 0 \text{ and } 1\leq i \leq r,\\
\phi(\alpha_1 ,\beta_1) & := & (\sigma_{r+1}, \sigma_{r+2}) &  \text{for }  g_0 = 1,\\
\phi(\alpha_1,\beta_1) & := & (\sigma_a, \tau_a), & \text{for }  g_0 \geq 2,\\
\phi(\alpha_2,\beta_2) & := & (\rho_1, \rho_2) , & \text{for } g_0 \geq 2, \text{ and}\\
\phi(\alpha_i)=\phi(\beta_i) & := & 1 & \text{for }  g_0 \geq 2 \text{ and } 3 \leq i \leq g_0,
\end{array}$$
such that
     \[
     \prod_{i = 1}^{r} \sigma_{i}[\phi(\alpha_1),\phi(\beta_1)][\rho_1, \rho_2]= 1, \tag{$\ast$}
      \]
We note that the existence of $\rho_1, \rho_2 \in A_n$ in ($\ast$) is guaranteed by Theorem \ref{thm:Ore}. For $1 \leq i \leq r$, we have
    $$\phi(\xi_i^{m_i})= \phi(\xi_i)^{m_i}=\sigma_i^{m_i}=1,$$
and when $g_0=0$, it follows from condition (iii)(b) (of Definition~\ref{defn:alt_data_set}) on $\D_a$ that $$\phi(\prod_{i = 1}^{r} \xi_{i})= \prod_{i = 1}^{r} \phi(\xi_{i})= \prod_{i = 1}^{r} \sigma_{i}=1.$$
Furthermore, when $g_0=1$, from condition (iv)(b) on $\D_a$, we get
$$\phi\left(\prod_{i = 1}^{r} \xi_{i}\cdot [\alpha_1,\beta_1]\right)= \prod_{i = 1}^{r} \phi(\xi_{i})\cdot [\phi(\alpha_1),\phi(\beta_1)] = \prod_{i = 1}^{r} \sigma_{i}\cdot[\sigma_{r+1},\sigma_{r+2}]=1,$$
and when $g_0\geq2$, it follows from ($\ast$) that
$$\phi\left(\prod_{i = 1}^{r} \xi_{i}\prod_{j = 1}^{g_{0}} [\alpha_{i},\beta_{i}]\right) = \prod_{i = 1}^{r} \phi(\xi_{i})\prod_{j = 1}^{g_{0}} \phi([\alpha_{i},\beta_{i}]) = \prod_{i = 1}^{r} \sigma_{i}[\phi(\alpha_1),\phi(\beta_1)][\rho_1, \rho_2]=1.$$

\noindent Therefore the map $\phi$ extends to a homomorphism for all $g \geq 0$. Moreover, by definition, it is apparent that $\phi$ preserves the order of torsion elements. For $g_0 \in \{0,1\}$, the surjectivity of $\phi$ follows from conditions (iii)(a) and (iv)(a) on $\D_a$, and for the case $g_0\geq 2$, since $\sigma_a, \tau_a \in \phi(\Gamma)$, it follows that $\phi$ is surjective. Finally, by condition (i) on $\D_a$, we have 
$$\frac{2-2g}{n!/2}=2-2g_0-\sum_{i=1}^{r} (1-\frac{1}{m_i}),$$
from which it follows that $\phi$ is the desired surface kernel epimorphism. 

Now, we define $\Theta_a([\D_a]): = [\phi]$ and show that $\Theta_a$ is well-defined.  Consider two equivalent data sets $$\D_a=\biggl(n, g_0; [\sigma_i, m_i; k_{1i},\dots , k_{l_ii}]_{i=1}^{r}\biggr) ~~\text{and} ~~\D_a'=\biggl(n, g_0; [\sigma_i', m_i'; k_{1i}',k_{2i}',\dots , k_{l_i'i}']_{i=1}^{r}\biggr),$$
where $\Theta_a([\D_a])=[\phi]$ and $\Theta_a([\D_a'])=\phi'$. Let $\phi:\Gamma \to A_n$ and $\phi':\Gamma' \to A_n$ are surface-kernel epimorphisms, where $\Gamma$ has the fixed presentation as in (\ref{eqn:eqn1}) and $\Gamma'$ has the presentation
\[\Gamma'=\biggl\langle~ \alpha_1',\beta_1',..., \alpha_{g_{0}}',\beta_{g_{0}}',\xi_{1}',...,\xi_{r}'~\bigg|~\xi_{1}'^{m_1'},...,\xi_{r}'^{m_r'},\prod_{i = 1}^{r} \xi_{i}'\prod_{j = 1}^{g_{0}} [\alpha_{i}',\beta_{i}']~\biggr\rangle.\tag{$2$} \label{eqn:eqn2}
\]
By Definition~\ref{defn:eq_data_sets}, there exists a permutation $\pi$ on $r$ letters such that exactly one of the following cases hold.

\noindent \textbf{Case (i).} $\sigma_i$ is conjugate to $\sigma_{\pi(i)}'$ in $A_n$ for all $i$.

\noindent \textbf{Case (ii)(a).} $\sigma_i$ is conjugate to $\sigma_{\pi(i)}'$ in $\Sigma_n$ for all $i$ and

\noindent \textbf{(b).} $\sigma_i$ is not conjugate to $\sigma_{\pi(i)}'$ in $A_n$ if the conjugacy class of $\sigma_i$ splits in $A_n$.

\noindent By Lemma \ref{lem:perm}, for the permutation $\pi$, there exist $h_1,h_2, \dots, h_r$ and $h$ such that: 
     $$h\circ\prod_{i = 1}^{r} {\xi_i}'\circ h^{-1}=\prod_{i = 1}^{r} h_i\circ \xi_{\pi(i)}'\circ h_{i}^{-1}.$$

We define a map $\psi$ by:
      \begin{center}
     $\begin{array}{ccll} 
     \psi(\xi_i) & := & h_i\circ \xi_{\pi(i)}'\circ h_{i}^{-1}, & \text{ for } 1 \leq i\leq r, \\
     \psi(\alpha_i) & := & h\circ \alpha_{i}'\circ h^{-1}, & \text{ for } 1\leq i\leq g_0, \text{ and }\\
     \psi(\beta_i) & := & h\circ \beta_{i}'\circ h^{-1}, & \text{ for }1\leq i\leq g_0.
\end{array}$
     \end{center}
Since we have \begin{eqnarray*}
       \psi \left(\prod_{i = 1}^{r} \xi_{i}\prod_{i = 1}^{g_{0}} [\alpha_{i},\beta_{i}]\right)
       &=&\prod_{i = 1}^{r} \psi(\xi_{i})\prod_{i = 1}^{g_{0}} [\psi(\alpha_{i}),\psi(\beta_{i})]\\
       &=&\prod_{i = 1}^{r} h_i\circ \xi_{\pi(i)}'\circ h_{i}^{-1}\prod_{i = 1}^{g_{0}} h\circ[\alpha_{i}',\beta_{i}']\circ h^{-1}\\
       &=&h\circ\prod_{i = 1}^{r} {\xi_{i}'}\circ h^{-1}h\circ\prod_{i = 1}^{g_{0}} [\alpha_{i}',\beta_{i}']\circ h^{-1}\\
       &=&h\circ\prod_{i = 1}^{r} \xi_{i}'\prod_{i = 1}^{g_{0}} [\alpha_{i}',\beta_{i}']\circ h^{-1}=h\circ 1\circ h^{-1}=1,
     \end{eqnarray*}     
$\psi$ extends to an isomorphism $\Gamma\to\Gamma'$. We now define an isomorphism $\tilde\chi:A_n\to A_n$ by $\tilde\chi=id$ for Case (ii)(a) and $\tilde\chi(\sigma)=(1~2)\circ\sigma\circ(1~2)$ for Case (ii)(b). Let $\epsilon_1,\epsilon_2 : A_n\hookrightarrow \homeo(S_g)$ be the actions uniquely determined by $\phi$ and $\phi'$, respectively. Take $H_i= \epsilon_i(A_n)$ for $i=1,2$, and consider $\chi: H_1\to H_2$ defined as $\chi=\epsilon_2\circ\tilde\chi\circ\epsilon_1^{-1}$. Since $\Gamma_i=\pi_1^{orb}(\Orb_{H_i})$ for $i=1,2$, $\phi_{H_1}=\epsilon_1\circ\phi$, and $\phi_{H_2}=\epsilon_2\circ\phi'$, we see that:
$$\phi_{H_2}\circ\psi(\xi_i)=\phi_{H_2}(h_i \circ \xi_{\pi(i)}'\circ h_{i}^{-1})=\phi_{H_2}(h_i) \circ \epsilon_2(\sigma_{\pi(i)}')\circ \phi_{H_2}(h_{i})^{-1}$$
and  $\chi\circ\phi_{H_1}(\xi_i) =(\epsilon_2\circ\tilde\chi\circ\epsilon_1^{-1})\circ(\epsilon_1\circ\phi)(\xi_i) = \epsilon_2\circ\tilde\chi(\sigma_i).$
Moreover, $\sigma_{\pi(i)}'$ is conjugate to $\tilde\chi(\sigma_i)$ in $A_n$ for both the cases, which implies $\epsilon_2(\sigma_{\pi(i)}')\in H_2$ and $\epsilon_2(\tilde\chi(\sigma_i))\in H_2$ are conjugate in $H_2$ for all $1\leq i \leq r$. Therefore, $H_1$ and $H_2$ satisfy the commutative diagram in Figure~\ref{fig:wkconjugacy}, which shows that $H_1$ and $H_2$ are weakly conjugate. Therefore, the map $\Theta_a$ is well-defined.

We now establish the surjectivity of $\Theta_a$. Let $H$ be a subgroup of $\homeo(S_g)$ and $\epsilon:A_n\hookrightarrow\homeo(S_g)$ be an embedding such that $\epsilon(A_n)=H$. Given this $H$-action, we want to construct an alternating data set of genus $g$ and degree $n$. Let $\phi_H:\pi_1^{orb}(\Orb_{H})\to H$ be the corresponding surface kernel epimorphism. Consider $\phi:\Gamma\to A_n$ defined by $\phi=\epsilon^{-1}\circ\phi_H$, where $\Gamma=\pi_1^{orb}(\Orb_{H})$. Then $\phi$ satisfies the hypothesis of Theorem \ref{thm:riemann}. Taking $\sig(\Gamma)=(g_0; m_1, \dots, m_r)$, and fixing a presentation of $\Gamma$ as in ($\ref{eqn:eqn1}$), we see that
$\sigma_i:=\phi(\xi_i) \in A_n$ and $\circ(\sigma_i)=m_i$. If $(k_{1i},\dots , k_{l_ii})$ is the cycle type of $\sigma_i$, then the tuple, $\D_a=\bigl(n, g_0; [\sigma_i, m_i; k_{1i},\dots , k_{l_ii}]_{i=1}^{r}\bigr)$ satisfies conditions (i)-(iv) of an alternating data set such that $\Theta_a([\D_a])=[\phi]$. Therefore, the surjectivity of $\Theta_a$ follows.

Finally, it remains to be shown that $\Theta_a$ is injective. In other words, if $H_1, H_2$ are two $A_n$-actions on $S_g$ that are weakly conjugate, then they must correspond under $\Theta_a$ to equivalent alternating data sets. Since $H_1, H_2$ are weakly conjugate, by definition, the diagram in Figure~\ref{fig:wkconjugacy} commutes up to conjugacy at the level of torsion elements. For $i=1,2$, let $\epsilon_i: A_n\hookrightarrow \homeo(S_g)$ be embeddings such that $\epsilon_i(A_n)=H_i$, and let $\phi_{H_i}:\pi_1^{orb}(\Orb_{H_i})\to H_i$ be the corresponding surface kernel epimorphisms. Consider $\phi_1:\Gamma\to\Sigma_n$ defined by $\phi_1=\epsilon_1^{-1}\circ\phi_{H_1}$ and $\phi_2:\Gamma'\to\Sigma_n$ defined by $\phi_2=\epsilon_2^{-1}\circ\phi_{H_2}$ , where $\Gamma=\pi_1^{orb}(\Orb_{H_1})$ and $\Gamma'=\pi_1^{orb}(\Orb_{H_2})$. Then, the $\phi_i$ satisfies the hypothesis of Theorem \ref{thm:riemann} for $i\in{1,2}$. Since $\epsilon_2^{-1}\circ\chi \circ \phi_{H_1}$ is another map satisfying the hypothesis Theorem \ref{thm:riemann} which is associated with the $H_1$-action, by uniqueness, it follows that $\epsilon_1^{-1}\circ\phi_{H_1}=\epsilon_2^{-1}\circ\chi\circ\phi_{H_1}$. Thus, we have $\epsilon_1^{-1}=\epsilon_2^{-1}\circ\chi$ since $\phi_{H_1}$ is surjective, which would imply that $\chi=\epsilon_2\circ\epsilon_1^{-1}$. Hence, the diagram in Figure \ref{fig:wkconjugacy} reduces to the following diagram, which commutes up to conjugacy at the level of torsion elements in $\Gamma$.
     \begin{center}
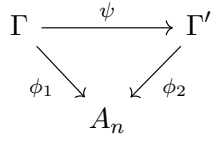

 \begin{tikzcd}[column sep=small]
 \Gamma \arrow{rr}{\psi} \arrow[swap]{dr}{\phi_1}& &\Gamma' \arrow{dl}{\phi_2}\\
 & A_n & 
 \end{tikzcd}
 \captionof{figure}{The reduced commutative diagram.}
    \end{center}

\noindent Since $\psi$ is an isomorphism, $\sig(\Gamma)= \sig(\Gamma')$. We fix the presentation of $\Gamma$ and $\Gamma'$ as in (\ref{eqn:eqn1}) and (\ref{eqn:eqn2}), respectively. Applying Theorem \ref{thm:zieschang}, $\psi(\xi_i)$ is conjugate to $\xi_{\pi(i)}'$ for some permutation $\pi$ on $r$ letters. Hence, $\phi_2\circ\psi(\xi_i)$ is conjugate to $\phi_2(\xi_{\pi(i)}')$. Therefore, we have $\sigma_i=\phi_1(\xi_i)$ is conjugate to $\phi_2(\xi_{\pi(i)}')=\sigma_{\pi(i)}'$. Hence, the data sets $\D_a$ and $\D_a'$ associated with the actions $\phi_1$ and $\phi_2$, respectively, are equivalent.
\end{proof}
In view of Proposition~\ref{prop:main1}, the weak conjugacy class of a $A_n$-action on $S_g$ will be represented by an alternating data set $\D_a$. Let $\epsilon:A_n\hookrightarrow \homeo(S_g)$ be an embedding associated with $\D_a$. Then, for simplicity of notation, we will denote the conjugacy class $D_{\epsilon(\sigma)}$, for each $\sigma \in \Sigma_n$, by $\D_a[\sigma]$. A careful application of Lemma~\ref{lem:subaction} yields the following.  
\begin{prop}
\label{prop:cyclicgen}
Let $\D_a =\bigl(n, g_0; [\sigma_i, m_i; k_{1i},\dots , k_{l_ii}]_{i=1}^{r}\bigr)$ represent the weak conjugacy class of an $A_n$-action on $S_g$. For $\sigma \in A_n$ with $|\sigma|=d$ , we have $$\D_a[\sigma]=\left(d,\tilde{g_0};(u_{ij}^{-1},t_i)^{[\mho_{\sigma,i,j}]}: u_{ij}\in\mathbb{Z}_{t_i}^{\times}, ~t_i~|~d\right),$$ 
where $\tilde{g_0}$ is determined by the  Riemann-Hurwitz equation, and $\mho_{\sigma,i,j}$ is given by the formula
$$\frac{d}{t_i}\cdot\mho_{\sigma,i,j}= \left|\mathbb{F}_{\sigma^{d/t_i}}(u_{ij},t_i)\right|- \sum\limits_{\substack{t_i'\in \N \\t_i'\neq t_i\\t_i|t_i'|d}}\sum\limits_{\substack{ (u_{i'j'},t_{i'})=1\\u_{ij}\equiv u_{i'j'}(\mathrm{mod}~t_i)}}\mho_{\sigma,i',j'}~.$$
\end{prop}
\noindent Thus, we obtain the first main result of this paper, which follows directly from Propositions~\ref{prop:main1} and~\ref{prop:cyclicgen}, and Remark~\ref{rem:weak_conj}.
\begin{theorem}[Main Theorem 1]\label{thm:main1}
Let $F, G \in \map(S_g)$ be two periodic mapping classes. Then $\langle F, G\rangle_w \cong A_n$ if and only if there exists an alternating data set $\D_a$ of genus $g$ and degree $n$ such that the cyclic data sets $D_{F}=\D_a[\sigma]$ and $D_{G}=\D_a[\tau]$ for some generating pair $\sigma,\tau \in A_n$.
\end{theorem}
\noindent In Theorem \ref{thm:main1}, we will refer to $\D_a[\sigma]$ and $\D_a[\tau]$ as the \textit{cyclic factors} associated with $\D_a$. If $\sigma=\sigma_a$ and $\tau=\tau_a$ are the standard generating pair (as defined at the beginning of Subsection~\ref{subsec:sym_alt_actions}), then the corresponding cyclic factors are called the \textit{standard cyclic factors} associated with $\D_a$.
\begin{exmp}[Icosahedral action]
\label{exmp:exmp1}
We know that the symmetry group of an icosahedron is $A_5$. Consider the surface $S_{19}$ embedded in $\mathbb{R}^3$ in the shape of an icosahedron shown in Figure \ref{fig:icosa}.
\begin{figure}[H]
\labellist
\tiny
\pinlabel $A$ at 154 279
\pinlabel $\frac{2\pi}{5}$ at 175 316
\pinlabel $\G$ at 110 316
\pinlabel $\frac{2\pi}{3}$ at 150 145
\pinlabel $\F$ at 110 145
\pinlabel $B$ at 154 4
\pinlabel $C$ at 142 85
\pinlabel $D$ at 188 167
\pinlabel $E$ at 88 167
\endlabellist
\centering
   \includegraphics[scale=.48]{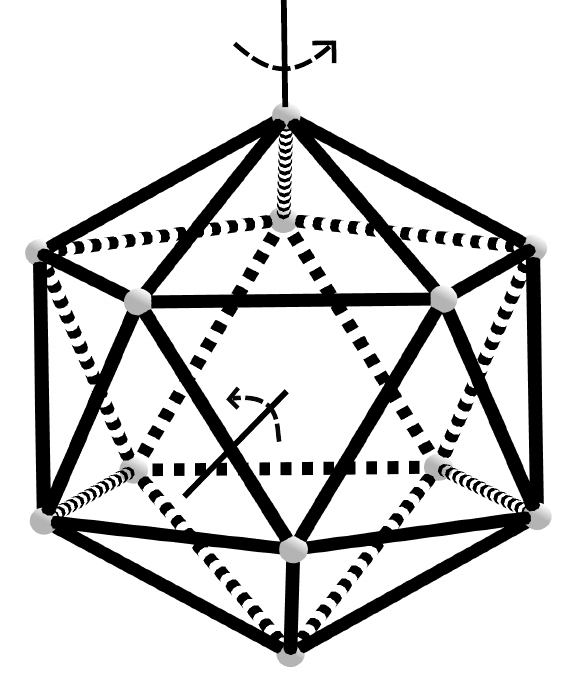}
  \caption{An $A_5$-action on $S_{19}$.}
  \label{fig:icosa}
\end{figure}
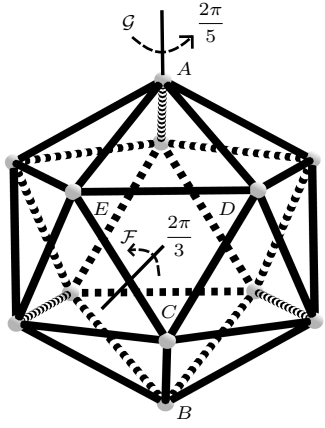
Let $\F \in \homeo(S_{19})$ be the free rotation by an angle $2\pi/3$ around the axis passing through the mid-points of the face $CDE$, and its opposite face and $\G$ be the $2\pi/5$ rotation around the axis passing through $A$ and $B$. Then $\F,\G$ generate an $A_5$-action on $S_{19}$ whose weak conjugacy class is represented by an alternating data set $$\D_a=(5,0;[(1~2)(3~4),2;2,2]^{[2]}, [(1~5~4~3~2),5;5], [(1~2~3~4~5),5;5]).$$
Moreover, we note that $D_F=\D_a[\sigma_a]=(3,7;-)$ and $D_G=\D_a[\tau_a]=(5,3;(1,5)^{[2]},(4,5)^{[2]})$
are also the standard cyclic factors of $\D_a$. 
\end{exmp}
\begin{exmp}[Dodecahedral action]
\label{exmp:exmp2}
Consider the surface $S_{11}$ embedded in $\mathbb{R}^3$ in the shape of a dodecahedron shown in Figure \ref{fig:icosa}. Since the dodecahedron is the dual of the icosahedron as a platonic solid, we obtain an $A_5$-action on $S_{11}$ (from the action in Example~\ref{exmp:exmp1}), as shown in Figure \ref{fig:dodeca}.
\begin{figure}[H]
\labellist
\tiny
\pinlabel $A$ at 138 0
\pinlabel $B$ at 198 25
\pinlabel $C$ at 245 70
\pinlabel $D$ at 206 100
\pinlabel $E$ at 134 60
\pinlabel $Z$ at 144 241
\pinlabel $\frac{2\pi}{5}$ at 83 250
\pinlabel $\F$ at 105 273
\pinlabel $\G$ at 20 230
\pinlabel $\frac{2\pi}{3}$ at 165 270
\pinlabel $Y$ at 92 202
\pinlabel $X$ at 163 44
\endlabellist
\centering
   \includegraphics[scale=.55]{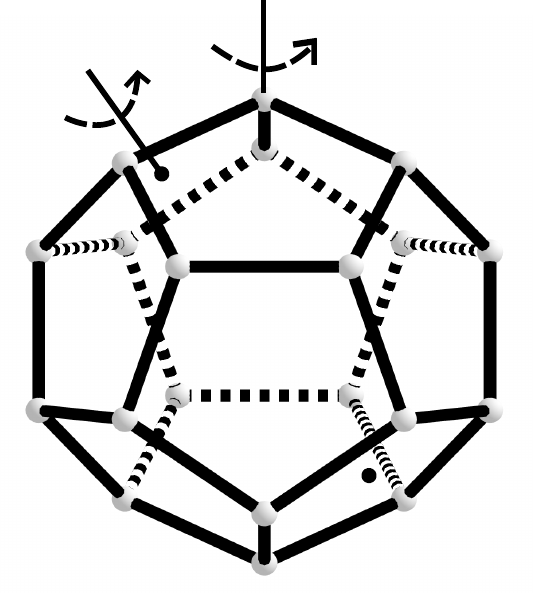}
  \caption{An $A_5$-action on $S_{11}$.}
  \label{fig:dodeca}
  \end{figure}
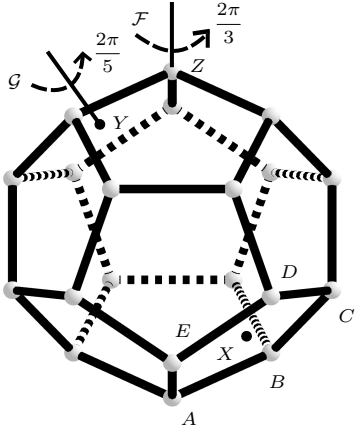
The weak conjugacy class of this action is represented by an alternating data set $\D_a=(5,0;[(1 3)(2 4),2;2,2]^{[2]},[(3 5 4),3;3],[(3 4 5),3;3])$, and the cyclic factors are given by cyclic data sets
$D_F=\D_a[\sigma_a]=[(3,3;(1,3)^{[2]},(2,3)^{[2]})$ and $D_G=\D_a[\tau_a]=(5,3;-)$.

\end{exmp}
\begin{exmp}[Cubic action]
\label{exmp:alt_cube}
By restricting the $\Sigma_4$-action described in Figure \ref{fig:cube} (of Section~\ref{sec:intro}), we get an $A_4$-action on $S_5$ whose weak conjugacy class is represented by the alternating data set $$\D_a=(4,0;[(1~2~3),3;3], [(1~2~3),3;3],[(2~3~4),3;3], [(2~3~4),3;3])$$
and the standard cyclic factors are given by cyclic data sets
$$\D_a[\sigma_a] =\D_a[\tau_a] =(3,1;(1~3)^{[2]},(2~3)^{[2]}).$$
\end{exmp}

\begin{exmp}[Octahedral action]
\label{exmp:alt_octa}
Consider the surface $S_{7}$ embedded in $\mathbb{R}^3$ in the shape of an octahedron shown in Figure \ref{fig:icosa}. Since the octahedron is the dual of the cube as a platonic solid, we also obtain an $\Sigma_4$-action on $S_{7}$ (from the $\Sigma_4$-action in Figure \ref{fig:cube} of Section~\ref{sec:intro}). By restricting this action, we get an $A_4$-action on $S_7$ whose weak conjugacy class is represented by the alternating data set $\D_a^1=(4,1;[(1~2)(3~4),2;2,2]^{[2]})$, and the standard cyclic factors are given by cyclic data sets $\D_a[\sigma_a] = \D_a[\tau_a] = (3,3;-).$
\end{exmp}
\begin{exmp}[Tetrahedral action]
Since the symmetry group of the tetrahedron is $A_4$, we get an $A_4$-action on $S_3$ as before, whose weak conjugacy class is represented by an alternating data set $$\D_a=(4,0;[(1~2)(3~4),2;2,2]^{[2]},[(2~4~3),3;3],[(2~3~4),3;3])$$
and the standard cyclic factors are given by cyclic data sets $\D_a[\sigma_a] = \D_a[\tau_a] = (3,1;(1,3),(2,3)).$
\end{exmp}

\subsection{Symmetric actions} We begin by defining a notion of a symmetric data set and an equivalence on these sets that is analogous to an alternating data set.
\begin{defn}\label{defn:sym_data_set}
A \textit{symmetric data set $\D_s$ of degree $n$ and genus $g\geq2$} is an ordered tuple $$(n, g_0; [\sigma_1, m_1; k_{11},\dots , k_{l_11}], \dots, [\sigma_r, m_r; k_{1r},\dots , k_{l_rr}]),$$ where $n,g_0$ are integers with $n\geq 3$ and $g_0 \geq 0$, and the $\sigma_i \in \Sigma_n$ satisfying the following conditions:
   \begin{enumerate}[(i)]
   \item  $$ \frac{2-2g}{n\, !/2}=  2-2g_{0}-\sum_{i=1}^{r} \left(1-\frac{1}{m_i}\right).$$
     \item $m_i$ is the order of $\sigma_i$ and $(k_{1i},\ldots,k_{l_ii})$ is the cycle type of $\sigma_i$.
\item For $g_0=0$, we have:
\begin{enumerate}[(a)]
    \item $\langle \sigma_1,\ldots,\sigma_i \rangle = \Sigma_n$ and 
    \item $\prod_{i = 1}^{r} \sigma_{i}= 1$.
\end{enumerate}
\item For $g_0=1$, there exist $\sigma_{r+1}, \sigma_{r+2} \in \Sigma_n$ such that:
\begin{enumerate}[(a)]
    \item $\langle \sigma_1,\ldots,\sigma_i, \sigma_{r+1}, \sigma_{r+2} \rangle = \Sigma_n$ and
    \item $\prod_{i = 1}^{r} \sigma_{i}= [\sigma_{r+2}, \sigma_{r+1}]$. 
    \end{enumerate}
\item For $g_0 \geq 2$, the permutation $\prod_{i = 1}^{r} \sigma_{i}$ is even.
 \end{enumerate}
\end{defn}
\begin{defn}
Two symmetric data sets $$\D_s=(n, g_0; [\sigma_1, m_1; k_{11},\dots , k_{l_11}], \dots, [\sigma_r, m_r; k_{1r},k_{2r},\dots , k_{l_rr}])$$ and $$\D_s'=(n, g_0; [\sigma'_1, m'_1; k_{11}',\dots , k'_{l_1'1}], \dots, [\sigma'_r, m'_r; k'_{1r},k'_{2r},\dots , k'_{l_r'r}])$$ are said to be \textit{equivalent} if there exists a permutation $\pi\in\Sigma_r$ such that $\sigma_i$ is conjugate to $\sigma_{\pi(i)}'$ in $\Sigma_r$, i.e.
\begin{enumerate}[(i)]
    \item $l_i=l_{\pi_i}'$ and
    \item $(k_{1i},\ldots , k_{l_ii})= (k'_{1\pi(i)},\ldots , k'_{l_i\pi(i)})$.
\end{enumerate}
\end{defn}
\noindent The proof of the following proposition, which allows us to encode the weak conjugacy classes of symmetric actions using symmetric data sets, is similar to Proposition~\ref{prop:main1}.
\begin{prop}
\label{prop:main2}
 Equivalence classes of symmetric data sets of genus $g$ and degree $n$ correspond to weak conjugacy classes of $\Sigma_n$-actions on $S_g$.
\end{prop}
\noindent We note that an analog of Proposition~\ref{prop:cyclicgen} also holds for a symmetric data set $\D_s$. However, for brevity, we refrain from explicitly stating it. This brings us to the second main result of this paper, whose proof is similar to Theorem~\ref{thm:main1}.
\begin{theorem}[Main theorem 2]
\label{thm:main2}
Let $F, G \in \map(S_g)$ be two periodic mapping classes. Then $\langle F, G\rangle_w \cong \Sigma_n$ if and only if there exists a symmetric data set $\D_s$ of genus $g$ and degree $n$, such that the cyclic data sets $D_{F}=\D_s[\sigma]$ and $D_{G}=\D_s[\tau]$ for some generating pair $\sigma,\tau \in \Sigma_n$.
\end{theorem}
\noindent In Theorem \ref{thm:main2}, $\D_s[\sigma]$ and $\D_s[\tau]$ are also called  \textit{cyclic factors} associated with $\D_s$. If $\sigma=\sigma_s$ and $\tau=\tau_s$ are the standard generating pair, then the corresponding cyclic factors are called the \textit{standard cyclic factors} associated with $\D_s$. 
\begin{exmp}
Consider the $\Sigma_4$-action on $S_5$ as described in Figure \ref{fig:cube}. The weak conjugacy class of this action can now be represented as $$\D_s=(4,0;[(1~2),2;2], [(3~4),2;2], [(1~2~3),3;3], [(2~3~4),3;3]),$$
and the standard cyclic factors associated with $\D_s$ are given by $D_F=\D_s[\sigma_s]=(2,4;(1~2)^{[4]})$ and $D_G=\D_s[\tau_s]=(4,2;-).$
\end{exmp}
\begin{exmp}
Consider the $\Sigma_4$-action on $S_7$ given by the symmetry group of octahedron. The weak conjugacy class of this action can be represented as $$\D_s=(4,0;[(3~4),2;2]^{[2]}, [(1~2~3~4),4;4], [(1~4~3~2),4;4]),$$
and the standard cyclic factors associated with $\D_s$ are given by $D_F=\D_s[\sigma_s]=(2,4;(1~2)^{[4]})$ and $D_G=\D_s[\tau_s]=(4,2;-).$
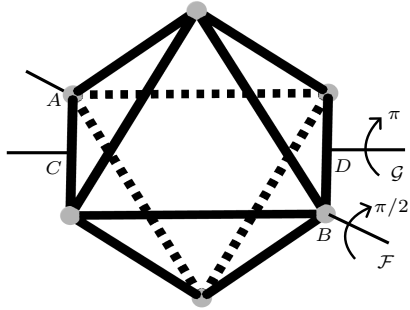
\begin{figure}[ht]
\labellist
\tiny
\pinlabel $A$ at 490 737
\pinlabel $B$ at 950 505
\pinlabel $\F$ at 1055 460
\pinlabel $\pi/2$ at 1065 552
\pinlabel $C$ at 490 620
\pinlabel $D$ at 982 623
\pinlabel $\G$ at 1075 610
\pinlabel $\pi$ at 1070 710
\endlabellist
\centering
   \includegraphics[scale=.23]{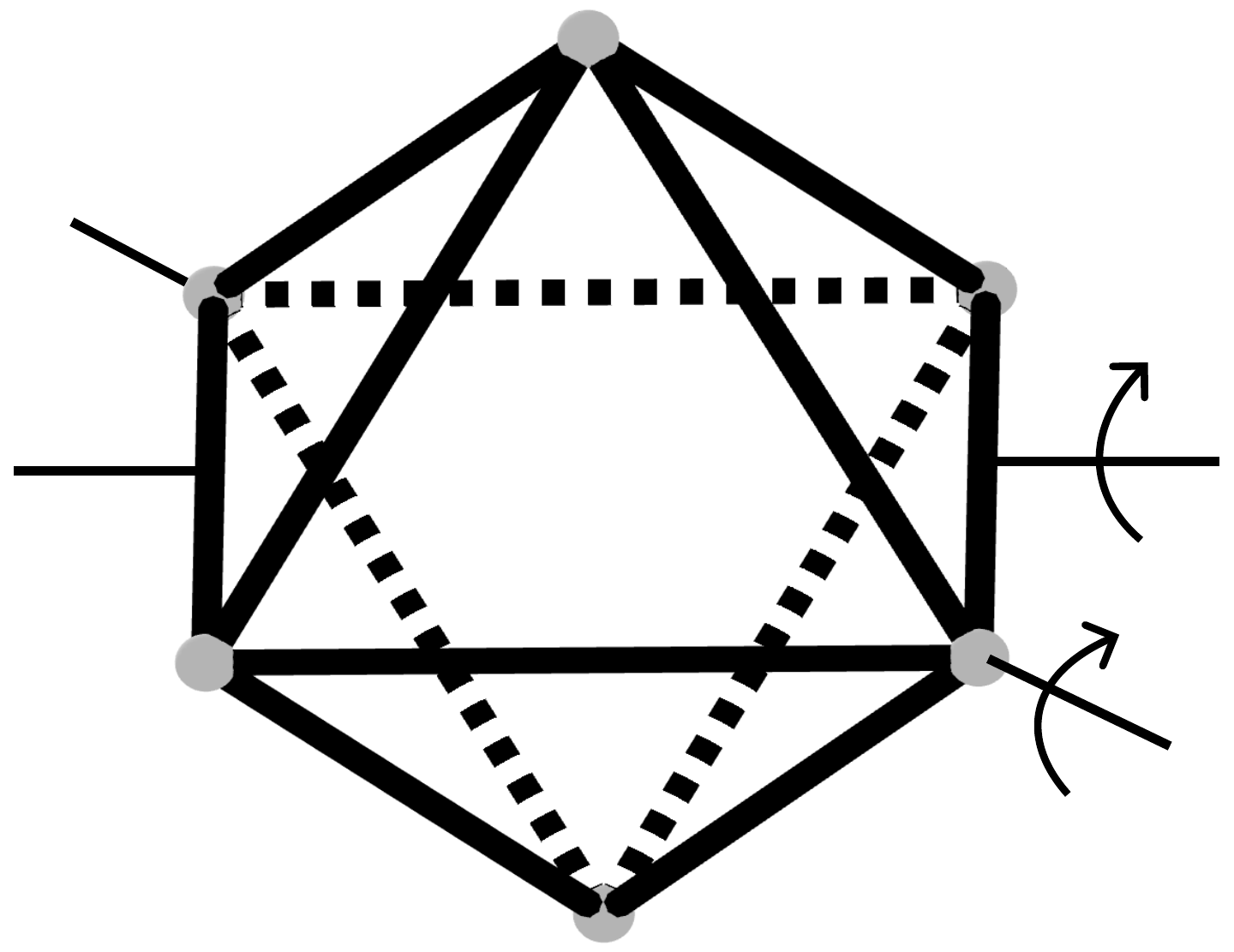}
  \caption{An $\Sigma_4$-action on $S_7$.}
  \label{fig:octahedron}
\end{figure}
It is worth mentioning that there is another $\Sigma_4$-action on $S_7$ represented by a symmetric data set $$\D_s=(4,0;[(1~2),2;2],[(2~3),2;2]^{[2]},[(3~4),2;2], [(1~2)(3~4),2;2,2]).$$

\end{exmp}
\noindent We now give an example of a non-free symmetric action whose component cyclic actions are free.   
\begin{exmp}
Consider a family of symmetric actions represented by $$\D_s=(n,1;[(1~2~3),3;3])~~\text{for}~~ n\geq5.$$
Then, the associated standard cyclic factors are given by $$\D_s[\sigma_s]=(2,\frac{6+n!}{6};-) ~\text{and}~ \D_s[\tau_s]=(n,\frac{3+(n-1)!}{3};-).$$
\end{exmp}
\noindent The following example provides a non-free symmetric action that is the (group) extension of a free $\Z_2$-action by a free alternating action.
\begin{exmp}\label{exmp:exmp4}
We know that $\Sigma_n= A_n \rtimes \langle(1~2)\rangle$ as an internal semi-direct product. Consider another family of symmetric actions given by: $$\D_s=(n,2;[(1~2)(3~4)(5~6),2;2,2,2]^{[2]})~~\text{for}~~n\geq6.$$
The associated alternating and cyclic data sets are given by $\D_a=(n,4;-)$ and $\D_s[\sigma_s]=(2,\frac{4+3\cdot n!}{4};-)$, respectively. 
\end{exmp}
\section{Lifting involutions under alternating covers}\label{section:involutions} Consider an $H$-action on $S_g$, where $H \cong A_n$, and the induced branched cover $p: S_g \to \Orb_H= S_g/H$. Since $A_n \lhd \Sigma_n$ and $[\Sigma_n:A_n] = 2$, every $\Sigma_n$-action on $S_g$ is realized by lifting an order-$2$ subgroup of $\map(\Orb_H)$ under $p$. (Note that the mapping classes $\map(\Orb_H)$ are represented by homeomorphisms that preserve the set of cone points in $\Orb_H$ and their orders.) Then the Dehn-Nielsen-Baer theorem also applies to this setting (see \cite{zieschang}), that is, the canonical map $\map(\mathcal{O}_H)\to \text{Out}(\pi_1^{orb}(\mathcal{O}_H))$ is an isomorphism. Thus, a natural question is when does an involution lift under the branched cover $p$. In this section, we give a complete answer to this question up to conjugacy. 

Let $\phi_H:\Gamma \to A_n$ be the surface kernel map associated with the $H$-action on $S_g$ and $\overline{G}_*\in \mathrm{Out}(\Gamma)$ be the map induced by $\overline{G}$, where $\Gamma \approx \pi_1^{orb}(\mathcal{O}_H)$ and $H \approx A_n$. 

\begin{defn}
\label{defn:wlp}
Let $\D_a$ be an alternating data set (as in Definition~\ref{defn:alt_data_set}), $D$ be a cyclic data set of degree $2$, and $\Pi \in \Sigma_r$. Then the pair $(\D_a, (D,\Pi))$ is called a \textit{weak-liftable pair} if there exists an $H \in [\D_a]$ and an involution $\overline{G} \in \map(\Orb_H)$ satisfying the following conditions.
\begin{enumerate}[(i)] 
\item $\overline{G}$ lifts under the branched cover $S_g \to S_g/H$.
\item If the orbifold genus of $\Orb_H$ is $g_0$, then $D_{\overline{G}} = D$, where $D_{\overline{G}}$ represents the conjugacy class of $\overline{G}$ in $\map(S_{g_0})$.
\item $\overline{G}_*(\xi_i) = \xi_{\Pi(i)}$, for $1 \leq i \leq r$, where the $\xi_i$ are the elliptic generators of $\pi_1^{orb}(\Orb_H)$ (as in the presentation (\ref{eqn:eqn1})).
\end{enumerate}
\end{defn}
\noindent Note that it is implicit from Definition~\ref{defn:wlp} that $|\Pi| \in \{1,2\}$. Also, since $\Pi$ depends on $\overline{G}$, we will fix the notation $\wlp$ for a weak-liftable pair with the implicit assumption that $\overline{G}$ is indeed the representative of $D_{\overline{G}}$ that lifts under $S_g \to S_g/H$. 

Consider a finite-sheeted regular branched cover $p: S_g \to \Orb_H$, and its deck transformation group $Deck(p)$. Then the following liftability criterion due to Broughton (see  \cite[Theorem 3.2.]{broughton_normalizer}), is an algebraic analog of the Birman-Hilden theorem~\cite{birman1971,margalit-winarski}, which asserts that the sequence $$1 \to Deck(p) \to \mathrm{SMod}(S_g) \to \mathrm{LMod}(\Orb_H) \to 1$$
is exact. 

\begin{prop}\label{prop:main4}
Let $\D_a$ be an alternating data set and $D_{\overline{G}}$ be a cyclic data set of degree $2$.  Then $\wlp$ forms a weak-liftable pair if and only if there exists $H \in [\D_a]$ and a $\chi_\omega\in \mathrm{Aut}(A_n)$ such that $\phi_H \circ \overline{G}_{\ast} = \chi_\omega \circ \phi_H$. 
\end{prop}

\noindent The next proposition follows directly from Theorem \ref{thm:zieschang} and Proposition \ref{prop:main4}.
\begin{prop}\label{prop:admissible_permutation}
Let $\wlp$ form a weak-liftable pair. Then
\begin{enumerate}[(i)]
\item $\sigma_i$ is conjugate to $\sigma_{\varPi_{\overline{G}}(i)}$ in $A_n$, if the conjugacy classes of $\sigma_i$ in $\Sigma_n$ and $A_n$ are equal and
\item $\sigma_i$ is conjugate to $\sigma_{\varPi_{\overline{G}}(i)}$ in $\Sigma_n$, but not in $A_n$, if the conjugacy classes of $\sigma_i$ in $\Sigma_n$ and $A_n$ are distinct.
\end{enumerate}
\end{prop}
\noindent As a consequence of Proposition \ref{prop:admissible_permutation}, we have the following.
\begin{cor}\label{cor:self_normalizing}
Let $\D_a = (n,0; [\sigma_i, m_i; k_{1i},\dots , k_{l_ii}]_{i=1}^{r})$ such that $\sigma_i$ is not conjugate to $\sigma_j$ in $\Sigma_n$ for $1\leq i,j\leq r$. Then any $H \in [\D_a]$ is self-normalizing in $\map(S_g)$. In particular, any $A_n$-action whose quotient orbifold has signature $(0;m_1,m_2,m_3)$ such that $m_1\neq m_2 \neq m_3$ does not extend either to a $\Sigma_n$-action or a $A_n \times \Z_2$-action.
\end{cor}
\begin{defn}
A weak-liftable pair  $\wlp$ is called a \textit{weak-liftable $\Sigma_n$-pair} (abbreviated as \textit{WLS-pair}) if there exists an $H \in [\D_a]$ such that the following conditions hold. 
\begin{enumerate}[(i)] 
\item $\overline{G}$ that lifts to a $G$ under the alternating cover $S_g \to S_g/H$ and
\item $\langle H\cup\{G\} \rangle \cong \Sigma_n$.
\end{enumerate}
\end{defn}
\begin{defn}
\label{defn:eq_wls_pairs}
Two WLS-pairs $\wlp$ and $(\D_a',(D_{\overline{G}'}, \Pi_{\overline{G}'}))$ are said to be equivalent if the following are satisfied:
\begin{enumerate}[(i)]
\item $\D_a$ and $\D_a'$ are weakly conjugate.
\item $D_{\overline{G}}=D_{\overline{G}'}$.
\item $\Pi_{\overline{G}'}=\pi^{-1} \circ \Pi_{\overline{G}} \circ \pi$,  where $\pi$ is as in Definition \ref{defn:eq_data_sets}.
\end{enumerate}
\end{defn}
\begin{prop}\label{prop:wls}There exists a well-defined surjective map:
\begin{center}
$\left\{\parbox{38mm}{\raggedright Equivalence classes of symmetric data sets}\right\}\xrightarrow{\makebox[1.5cm]{$\Psi$}}
\left\{\parbox{34mm}{\raggedright Equivalence classes\\ of WLS-pairs} \right\}.$
\end{center}
\end{prop}
\begin{proof}
Let $\D_s=(n,g_0; [\sigma_i, m_i; k_{1i},\dots , k_{l_ii}]_{i=1}^{r})$ be a symmetric data set of genus $g$. Let $H<\homeo(S_g)$ be the corresponding $\Sigma_n$-action given by Proposition \ref{prop:main2}. Since there exists a unique $H'\triangleleft H$ such that $H'\cong A_n$. Thus, we have an $A_n$-action on $S_g$ and a $\mathbb{Z}_2$-action on $S_g/H'$ induced by $H/H'~(=\langle \overline{G} \rangle)$, as shown in the diagram below:
\begin{center}
     \begin{tikzcd}  
S_g
\arrow[rr, bend right, "/H~(\cong\Sigma_n)" ']
\arrow[r, "{/H'~(\cong A_n)}"] &[3em] \mathcal{O}_{H'} \arrow[r, "/\langle \overline{G} \rangle~(\cong\Z_2)"] &[3em] \mathcal{O}_{H}
     \end{tikzcd}
     \end{center}
Hence, by Proposition \ref{prop:main1}, we obtain an alternating data set $\D_a$ corresponding to $H'$-action on $S_g$.

Consider $\Pi_{\overline{G}} \in \Sigma_r$ be the permutation induced by $\overline{G}$ and $D_{\overline{G}}$ denote the cyclic data set of  involution induced by $\overline{G}$ on $S_g/A_n \approx S_{g_0'}$. Then, it follows from Definition~\ref{defn:eq_wls_pairs} that the mapping $\psi([\D_s]):= [\wlp],$ is well-defined surjection. We will now compute $\D_a$ and $D_{\overline{G}}$ explicitly. 

Let $\Gamma$ be as in Equation (\ref{eqn:eqn1}). Let $\phi_H:\Gamma\to\Sigma_n$ be a surface kernel map associated $\D_s$ that satisfies $\phi(\xi_i)=\sigma_i$ for $1\leq i\leq r$. Then taking $\Gamma'=\phi_H^{-1}(A_n)$ and $\Gamma_0=\phi_H^{-1}(1)$, we see that $\Gamma,\Gamma_0 \triangleleft \Gamma$ and $\Gamma/\Gamma'\cong (\Gamma/\Gamma_0)/(\Gamma'/\Gamma_0)\cong \Sigma_n/A_n.$ Hence, $\Gamma'\triangleleft\Gamma$ is a unique index two subgroup, and the corresponding surface kernel map for the $H'$-action is given by $\phi_{H'}:\Gamma'\to A_n$.

It is well known that if $\gamma$ is an elliptic generator of $\Gamma'$, then $|\gamma|= c_i/t_i$, where $c_i=|\xi_i|$ and $t_i=\circ(\xi_i\Gamma')$ in $\Gamma/\Gamma'$, for some unique $i$. Furthermore, the conjugacy class of $\gamma$ in $\Gamma$ splits into $[\Gamma:\Gamma']/t_i$ many conjugacy classes in $\Gamma'$. We note that $c_i=|\xi_i|= |\sigma_i|=m_i$, $t_i=|\xi_i\Gamma'|$ in $\Gamma/\Gamma'$, $t_i=|\sigma_iA_n|$ in $\Sigma_n/A_n$, and $[\Gamma:\Gamma']=2.$ Thus, every odd permutation $\sigma_i$ in $\D_s$ contributes an even permutation in $\D_a$ that lies in the conjugacy class of $\sigma_i^2$ in $A_n$. However, when $\sigma_i^2=1$, there is no such contribution. Moreover, when the conjugacy classes of $\sigma_i$ in $\Sigma_n$ and $A_n$ and equal, every even permutation $\sigma_i$ in $\D_s$ contributes two even permutations in $\D_a$ that lie in the conjugacy class of $\sigma_i$ in $\Sigma_n$ (or $A_n$).  But, when the conjugacy class of $\sigma_i$ (in $\Sigma_n$) disintegrates in $A_n$, it contributes two non-conjugate permutations (of $A_n$) in $\D_a$ that lie in the conjugacy class of $\sigma_i$ in $\Sigma_n$. Thus, $\D_a$ is determined up to equivalence as in Definition \ref{defn:eq_data_sets}. Finally, by a straightforward computation we obtain that $D_{\overline{G}}=(2,g_0;(1,2),\overset{\ell}{\cdots},(1,2)),$ where $\ell$ is the number of odd permutations in $\D_s$ and $g_0$ is determined by the Riemann-Hurwitz equation $$\frac{2-2g}{2}=2-2g_0-\frac{\ell}{2}.$$
\end{proof}
\noindent An immediate consequence of Proposition \ref{prop:wls} is the following.
\begin{theorem}\label{thm:wls}
 $\wlp$ forms a WLS-pair if and only if there exists a symmetric data set $\D_s$ such that $\Psi([\D_s])= [\wlp]$.
\end{theorem}
\begin{cor}\label{cor:cor3}
There exists a free $A_n$-action on $S_g$ if and only if there exists $k\in \mathbb{N}$ such that $g= 1+k\cdot\frac{n!}{2}$. Furthermore:
\begin{enumerate}[(i)]
\item A free $A_n$-action on $S_g$ extends to a free $\Sigma_n$-action on $S_g$ if and only if $k$ is even, and 
\item A free $A_n$-action on $S_g$ extends to a non-free $\Sigma_n$-action on $S_g$ if $k \geq 3$ and $k$ is odd.
\end{enumerate}
\end{cor}
\begin{proof}
Note that any free $A_n$-action on $S_g$ is represented by the alternating data set
$\D_a=(n, g_0; -)$ for some $g_0\geq 2$, satisfying the Riemann-Hurwitz equation $\frac{2-2g}{n!/2}=2-2g_0$, which implies that $g= 1+ \frac{n!}{2}(g_0-1)$, i.e. $ g= 1+k\cdot\frac{n!}{2}$ for some positive integer $k:=g_0-1$.

Further, assume that there exists a free $A_n$-action on $S_g$, which can be extended to a free $\Sigma_n$-action.
Then there exists a free $\Z_2$-action on $S_{g_0}$ by Theorem ~\ref{thm:wls}. Hence, $g_0$ must be odd, i.e., $k$ is even. Conversely, if $k$ is even then any free alternating actions with associated $\D_a=(n, g_0; -)$ can be extended to a free $\Sigma_n$-action as $\Psi([\D_s])=[\wlp]$ where $\D_s=(n, \tilde{g_0}; -)$ with $\tilde{g_0}=1+k/2$, $D_{\overline{G}}=(2,\tilde{g_0};)$, and $\Pi_{\overline{G}}= id$.

Moreover, for odd $k\geq 3$, one can check that $(\D_a,(D_{\overline{G}'},\Pi_{\overline{G}'}))$ forms a WLS-pair with $D_{\overline{G}'}=(2,\tilde{g_0};(1,2),(1,2))$ and $\Pi_{\overline{G}'}= id$.
\end{proof}
\noindent Note that for $k=1$ in Corollary \ref{cor:cor3}, it boils down to whether $(1;2,2)$ is an actual signature of $\Sigma_n$ or not, which is unknown to the authors! Now we will show that any involution on $S_{g_0}$ having quotient orbifold genus $\geq2$, weakly lifts under a free alternating cover: $S_g \overset{/A_n}{\longrightarrow} S_{g_0}$.
\begin{cor}
Let $\D_a=(n,g_0;-)$ and $D_{\overline{G}}=(2,\overline{g}_0;(1,2),\overset{\ell}{\cdots},(1,2))$ be of genus $g_0$ such that $\Pi_{\overline{G}}=id$ and $\overline{g}_0 \geq 2$. Then $\wlp$ forms a WLS-pair, and hence a weak-liftable pair.
\end{cor}
\begin{proof}
Consider the tuple $(n,\overline{g}_0;[(1~2),2;2]^{[\ell]})$. Then by Definition~\ref{defn:cyclic_data_set}, $\ell$ must be even. Hence, $\D_s:=(n,\overline{g}_0;[(1~2),2;2]^{[\ell]})$ satisfies the the conditions of a symmetric data set, since $\prod_{i = 1}^{\ell} (1~2)$ is an even permutation. Finally, It follows from Proposition \ref{prop:wls} that $\Psi([\D_s])=[\wlp]$.
\end{proof}
\begin{rem}
\label{rem:extension}
It is well-known~\cite[Chapter 11]{Robinson} that when $n \neq 6$, any extension of $\Z_2$ by $A_n$ is isomorphic to either $\Sigma_n$ or $A_n \times \Z_2$. In other words, the short exact sequence $$1\to A_n \to H \to \Z_2 \to 1$$ always splits when $n \neq 6$.
\end{rem}
\noindent Thus, we obtain the following result, which follows directly from Theorem~\ref{thm:wls} and Remark~\ref{rem:extension}.
\begin{theorem}\label{thm:alternating_extensions}
For $n \neq 6$, consider an $H<\map(S_g)$ such that $H\cong A_n$ and let $\D_a$ be an alternating data set representing the weak conjugacy class of the $H$-action. Then there exists an $H'<\map(S_g)$ that is weakly conjugate to $H$ and an $H''<\map(S_g)$ with $[H'':H']=2$ if and only if  $\wlp$ is a weak-liftable pair.  Furthermore:
\begin{enumerate}[(i)]
\item $H''$ can be chosen to be isomorphic to $\Sigma_n$ if and only if  $\wlp$ forms a WLS-pair and
\item $H''$ can be chosen to be isomorphic to $A_n \times \Z_2$ if $\wlp$ is a weak-liftable pair that is not WLS-pair.
\end{enumerate} 
\end{theorem}
\begin{exmp} 
For $n \geq 6$, consider the free alternating action $\D_a=(n,4;-)$ from Example ~\ref{exmp:exmp4}. Taking $D_{\overline{G}}=(2,2;(1,2),(1,2))$ and $\Pi_{\overline{G}}= id$, we consider two non-equivalent symmetric data sets $\D_s=(n,2;[(1~2),2;2]^{[2]}$) and $\D_s'=(n,2;[(1~2)(3~4)(5~6),2;2,2,2]^{[2]}).$
It can be easily seen that $\Psi([\D_s]) = \Psi([\D_s']) = [\wlp]$. Hence, there can exist conjugate liftable involutions $\overline{G}$ and $\overline{G}'$ under the branched cover $S_{g} \to S_{g}/A_n$ with $\Pi_{\overline{G}}=\Pi_{\overline{G}'}$ that lift to symmetric actions that are not weakly conjugate.
\end{exmp}
\begin{exmp}[Icosahedral action]
Consider the $A_5$-action on $S_{19}$ from Example \ref{exmp:exmp1} with weak conjugacy class represented by 
\begin{center}
$\D_a=(5,0;[(1~2)(3~4),2;2,2]^{[2]}, [(1~5~4~3~2),5;5], [(1~2~3~4~5),5;5])$.
\end{center}
\noindent Since $\Orb_H$ is a sphere for any $H \in [\D_a]$, any action on $\Orb_H$ has cyclic data set $D=(2,0;(1,2)^{[2]})$. By Proposition~\ref{prop:admissible_permutation}, any liftable involution can only induce permutations $\Pi_{\overline{G}_1}=(1~2)$, $\Pi_{\overline{G}_2}=(3~4)$, and $\Pi_{\overline{G}_3}=(1~2)(3~4)$ in $\Sigma_4$, and the quotient orbifolds of these actions have signatures $(0;2,10,10)$, $(0;4,4,5)$, and $(0;2,2,2,5)$, respectively. It can be verified that the signature $(0;2,10,10)$ does not form a symmetric data set of degree $5$ and genus $19$. Furthermore, by Proposition~\ref{prop:main1}, the signatures $(0;4,4,5)$ and $(0;2,2,2,5)$ give rise to unique symmetric data sets, up to equivalence, given by
$\D_s = (5,0;[(1~5~2~4),4;4], [(2~4~3~5),4;4], [(1~2~3~4~5),5;5])$ and 
$\D_s' = (5,0;[(1~4),2;2], [(1~5),2;2], [(1~3)(2~4),2;2,2], [(1~2~4~5~3),5;5]),$
respectively, such that
$\Psi([\D_s])=[(\D_a,(D_{\overline{G}_2}, \Pi_{\overline{G}_2}))], \, \Psi([\D_s'])= [(\D_a,(D_{\overline{G}_3}, \Pi_{\overline{G}_3}))]$, and $D_{\overline{G}_2}=D_{\overline{G}_3}=D$. The two $\Z_2$-actions $D_{\overline{G}_2}$ and $D_{\overline{G}_3}$ on $\Orb_H$ are shown in Figure \ref{fig:fig4}, where labeled pair $(z,n)$ represents a cone point $z$ along with its order $n$.
\begin{figure}[ht]
\begin{subfigure}{0.4\textwidth}
\labellist
\small
\pinlabel $(x_1,2)$ at 248 410
\pinlabel $\pi$ at 210 480
\pinlabel $\G_2$ at 120 480
\pinlabel $(x_2,2)$ at 248 40
\pinlabel $(x_3,5)$ at -65 250
\pinlabel $(x_4,5)$ at 415 250
\endlabellist
\centering
   \includegraphics[scale=.23]{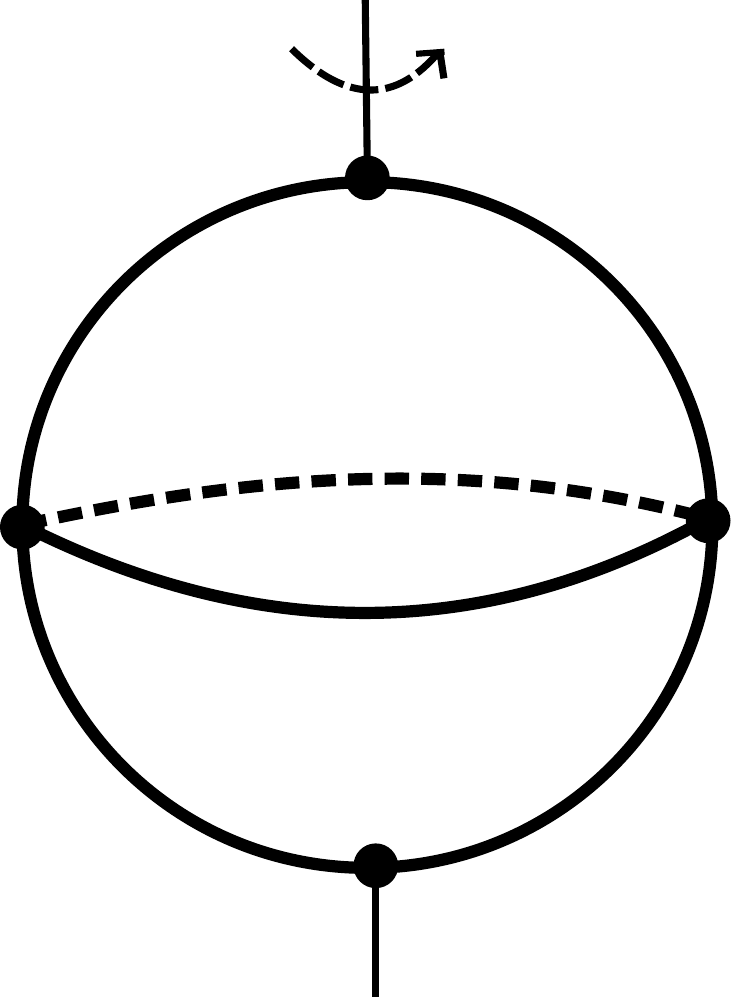}
\end{subfigure}
\begin{subfigure}{0.4\textwidth}
\labellist
\small
\pinlabel $(y_1,2)$ at 242 280
\pinlabel $(y_2,2)$ at 178 152
\pinlabel $\pi$ at 210 480
\pinlabel $\G_3$ at 120 480
\pinlabel $(y_3,5)$ at -65 210
\pinlabel $(y_4,5)$ at 410 210
\endlabellist
\centering
   \includegraphics[scale=.23]{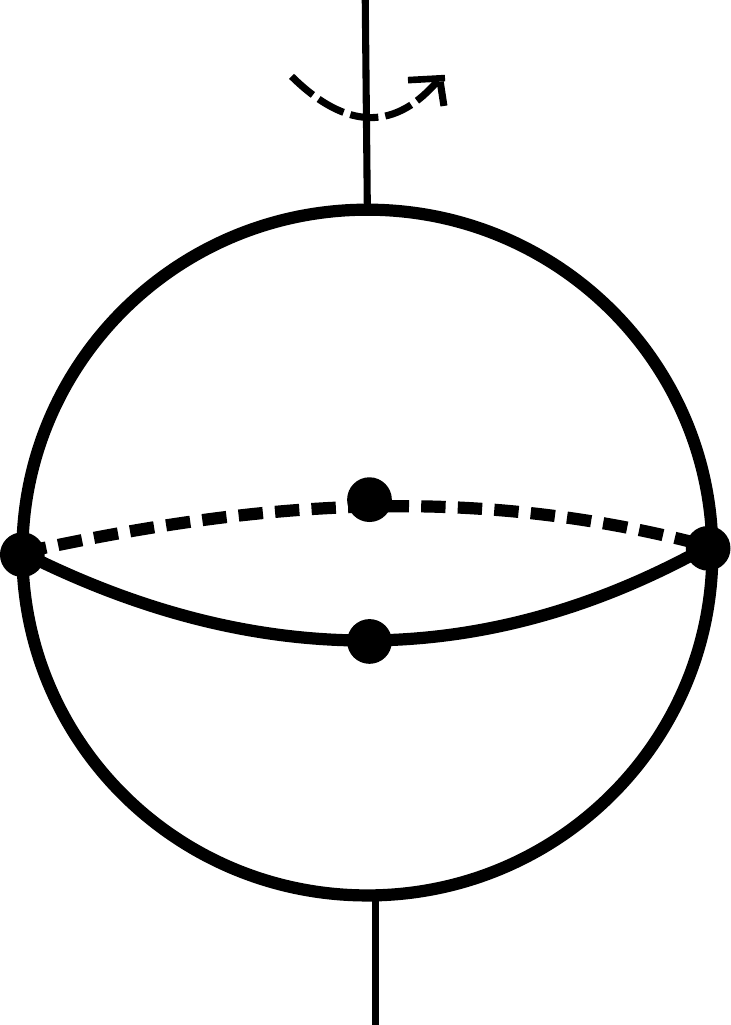}
\end{subfigure}
\caption{The two $\Z_2$-actions on $\Orb_H$ induced by $D_{\overline{G}_2}$(left) and $D_{\overline{G}_3}$(right).}
\label{fig:fig4}
  \end{figure}
\end{exmp}
\begin{exmp}[Dodecahedral action]
Consider the $A_5$-action on $S_{11}$ as in Example \ref{exmp:exmp2} represented by $\D_a=(5,0;[(1 3)(2 4),2;2,2]^{[2]},[(3 5 4),3;3],[(3 4 5),3;3])$.  By Proposition~\ref{prop:admissible_permutation}, any liftable involution can only induce permutations $\Pi_{\overline{G}_1}=(1~2)$, $\Pi_{\overline{G}_2}=(3~4)$, and $\Pi_{\overline{G}_3}=(1~2)(3~4)$ in $\Sigma_4$, whose quotient orbifolds have signatures $(0;2,6,6)$,$(0;3,4,4)$, and $(0;2,2,2,3)$, respectively. It can be verified that there exists no $\Sigma_5$-action on $S_{11}$ with quotient orbifold signature $(0;2,2,2,3)$. However, Proposition~\ref{prop:main1}, the signatures $(0;2,6,6)$ and $(0;3,4,4)$ give rise to unique symmetric data sets, up to equivalence, given by $$\D_s = (5,0;[(1~4)(2~3~5),6;2,3],[(1~2)(3 ~4~5),6;2,3],[(1~3)(2~4),2;2,2)$$ and $\D_s' = (5,0;[(1~4~3~2),4;4],[(1~2~3~5),4;4],[(3~4~5),3;3)$, respectively, such that
$\Psi([\D_s])=[(\D_a,(D_{\overline{G}_1}, \Pi_{\overline{G}_1}))]$ and $\Psi([\D_s'])= [(\D_a,(D_{\overline{G}_2}, \Pi_{\overline{G}_2}))]$, and $D_{\overline{G}_1}=D_{\overline{G}_2}=(2,0;(1,2)^{[2]})$.
\end{exmp}
\begin{exmp}[Cubic action]
Consider the $A_4$-action on $S_{5}$ as in Example \ref{exmp:alt_cube} represented by
\begin{center}
$\D_a=(4,0;[(1~2~3),3;3],[(1~2~3),3;3],[(2~3~4),3;3],[(2~3~4),3;3])$.
\end{center}
Any liftable involution can only induce permutations $\Pi_{\overline{G}_1}=(1~3)$, $\Pi_{\overline{G}_2}=(1~3)(2~4)$ in $\Sigma_4$ which amount to signatures $(0;3,6,6)$ and $(0;2,2,3,3)$, respectively. But there exist no $\Sigma_4$-action on $S_7$ with quotient orbifold signature $(0;3,6,6)$. However, by Proposition~\ref{prop:main1}, the signature $(0;2,2,3,3)$ gives rise to a unique symmetric data set, up to equivalence, given by $\D_s = (4,0;[(1~2),2;2],[(3~4),2;2],[(1~2~3),3;3],[(2~3~4),3;3])$ such that $\Psi([\D_s])=[(\D_a,(D_{\overline{G}_2},\Pi_{\overline{G}_2}))]$, where $D_{\overline{G}_2}=(2,0;(1,2)^{[2]})$.
\end{exmp}
\begin{exmp}[Octahedral action]
Consider the $A_4$-action on $S_7$ as in Example \ref{exmp:alt_octa} represented by $\D_a^1=(4,1;[(1~2)(3~4),2;2,2]^{[2]})$. Any involution in $\map(\Orb_H)$ for $H \in \D_a^1$ has a conjugacy class represented by either $D_1=(2,0;(1,2)^{[4]})$ or $D_2=(2,1;-)$. Any liftable involution in the class $D_1$ can only induce permutations $\Pi_{\overline{G}_1}= id$ and $\Pi_{\overline{G}_2}=(1~2)$ in $\Sigma_2$ whose quotient orbifolds have signatures $(0;2,2,4,4)$ and $(0;2,2,2,2,2)$, respectively. These signatures give rise to unique symmetric data sets given by $$\D_s = (4,0;[(1~2~3~4),4;4],[(1~4~3~2),4;4],[(2~3),2;2],[(2~3),2;2])$$ and $\D_s' = (4,0;[(1~2),2;2],[(2~3),2;2],[(2~3),2;2],[(3~4),2;2],[(1~2)(3~4),2;2,2])$, respectively, such that
$\Psi([\D_s])=[(\D_a,(D_{\overline{G}_1}, \Pi_{\overline{G}_1}))]$ and $\Psi([\D_s'])= [(\D_a,(D_{\overline{G}_2}, \Pi_{\overline{G}_2}))]$, and $D_{\overline{G}_1}=D_{\overline{G}_2}= D_1$.

Now, any liftable involution with associated cyclic data set $D_2$ can only induce the permutation $\Pi_{\overline{G}_3}= id$ whose quotient orbifold signature is $(1;2)$. But there exists no $\Sigma_4$-action on $S_7$ with quotient orbifold signature $(1;2)$. Hence $(\D_a^1,(D_{\overline{G}_3},\Pi_{\overline{G}_3}))$ with $D_{\overline{G}_3}=D_2$ does not form a WLS-pair. But, it can be verified using Proposition \ref{prop:main4} that $(\D_a^1,(D_{\overline{G}_3},\Pi_{\overline{G}_3}))$ forms a weak-liftable pair, resulting in an $A_4\times \Z_2$-action on $S_7$.
\end{exmp}

\begin{exmp}
Consider another $A_4$-action on $S_7$ represented by $$\D_a^2=(4,0;[(1~2)(3~4),2;2,2]^{[2]},[(2~3~4),3;3]^{[3]}).$$ Any involution in $\map(\Orb_H)$ for $H \in \D_a^2$ has a unique cyclic data set $D_{\overline{G}}=(2,0;(1,2)^{[2]})$ and an induced permutation $\Pi_{\overline{G}}= (1~2)(3~4) \in \Sigma_5$ with quotient orbifold signature $(0;2,2,3,6)$. One can check, as before, that $(\D_a^2,(D_{\overline{G}},\Pi_{\overline{G}}))$ does not form a WLS-pair but it forms a weak-liftable pair that results in an $A_4\times \Z_2$-action on $S_7$.
\end{exmp}

\section{Embeddable cyclic actions inside $\Sigma_n$ and $A_n$-action}\label{section:embedding} In this section, we show that a hyperelliptic involution or an irreducible periodic mapping class is not contained in a symmetric (or an alternating) subgroup of $\map(S_g)$. To begin with, we have the following technical proposition that provides an asymptotic bound for the order of an element in such a subgroup.
\begin{prop}\label{prop:bound}
Let $H\leq \mathrm{Mod}(S_g)$ such that $H\cong \Sigma_n \text{(or } A_n)$. Then, for sufficiently large $n$, any element in $H$ has order less than $g/k$, for some $k\in \mathbb{N}$.
\end{prop}
\begin{proof}
Let $H\cong \Sigma_n$, and let $g(n)$ denote the largest order of an element in $\Sigma_n$. Then it is known that $g(n) < e^{n/e}$. Then it easily seen that the sequence $f(n):=\frac{n!}{e^{n/e}}$ diverges to $+\infty$. Since, $\frac{n!}{g(n)}>\frac{n!}{e^{n/e}}=f(n)$, it follows that $\Bigl\{\frac{n!}{g(n)}\Bigr\}$ also diverges to $+\infty$.

On the contrary, suppose we assume that there exists an $F \in H$ such that $|F|>g/k$, for some fixed $k$. Since $|H|\leq 84(g-1)$ (see~\cite{hurwitz}), we have  
    $$\frac{n!}{g(n)}\leq \frac{|H|}{|F|} < \frac{84(g-1)}{g/k} <84k,$$
which contradicts the fact that $\Bigl\{\frac{n!}{g(n)}\Bigr\}$ diverges to $+\infty$. Therefore, we have $|F|\leq g/k$ for sufficiently large $n$. An analogous argument also works when $H\cong A_n$.
\end{proof}
\begin{cor}
For $n\geq5$, let $H\leq \mathrm{Mod}(S_g)$ such that $H\cong \Sigma_n$ (or $A_n$). Then, there is no irreducible mapping class in $H$.
\end{cor}
\begin{proof}
Let $F \in \mathrm{Mod}(S_g)$ be an irreducible mapping class. Then, it directly follows from the Riemann-Hurwitz Equation that $|F|> 2g$. Hence, by Proposition \ref{prop:bound}, it follows that $F \notin H$ for $n \geq 5$.
\end{proof}

\begin{prop}
Let $H\leq \mathrm{Mod}(S_g)$ such that $H\cong \Sigma_n$ (or $A_n$) for $n\geq8$. Then $H$ does not contain hyperelliptic involution.
\end{prop}
\begin{proof}
Suppose that $H\cong \Sigma_n$ with weak conjugacy class represented by a symmetric data set $\D_s$ as in Definition~\ref{defn:sym_data_set}. Since the hyperelliptic involution has quotient orbifold genus zero, it is clear that if the orbifold $S_g/H$ has positive genus, then no hyperelliptic involution can be inside $H$. Hence, we shall assume from here on that $g_0=0$. Then, the Riemann-Hurwitz equation would imply that
$$2+2g=4+n!\biggl(r-2-\sum_{i=1}^{r}\frac{1}{m_i}\biggr).$$
Let $X$ denote the expression in the right-hand side of the preceding equation and denote the hyperelliptic involution $\sigma\in \mathrm{Mod}(S_g)$, we have $D_{\sigma}=\left(2,0; (1,2)^{[2g+2]}\right)$. Now, let $$Y=|C_{\Sigma_n}(\sigma)| \cdot \sum\limits_{\substack{ \sigma\sim_{\Sigma_n} \sigma_i\\1\leq i\leq r}}\frac{1}{m_i}$$ in the notation of Proposition~\ref{prop:cyclicgen}. To rule out the possibility of $g_0=0$, it suffices to establish that $X>Y$ for $n\geq8$, since both the expessions compute the number of fixed points of $\sigma$. We break the argument into the following three cases.

\noindent\textbf{Case (i): $r\geq 5$.} Let $X_1:=n!(\frac{r}{2}-2)$ and $Y_1:= 2(n-2)!\cdot(r-2)$. Then it follows that $X>X_1>Y_1>Y$ since $n(n-1)(\frac{r}{2}-2)>2(r-2).$

\noindent\textbf{Case (ii): $r=4$.} Let $X_1:=n!\left[4-2-\left(\frac{1}{2}+\frac{1}{2}+\frac{1}{2}+\frac{1}{3}\right)\right]=\frac{n!}{6}$ and $Y_1:= 4(n-2)!$. As $n(n-1)>24$, we have $X>X_1>Y_1>Y$. 

\noindent\textbf{Case (iii): $r=3$.} Let $X_1:=4+n!\left[3-2-\left(\frac{1}{2}+\frac{1}{3}+\frac{1}{7}\right)\right]=\frac{n!}{42}$ and $Y_1:= (n-2)!$. Since $n(n-1)>42,$ we have $X>X_1>Y_1>Y$.
\end{proof}

\section{Classification of the weak conjugacy classes in $\map(S_{10})$ and $\map(S_{11})$.}\label{section:table} In this section, we will classify the weak conjugacy classes of the alternating and symmetric subgroups of $\map(S_{10})$ and $\map(S_{11})$ using Theorems \ref{thm:main1} and \ref{thm:main2}. Our choice of $S_{10}$ and $S_{11}$ is motivated by the fact that these surfaces admit a variety of alternating and symmetric actions up to weak conjugacy. With slight abuse of notation, we will denote $\sigma,\tau$ as the standard generating pair for both $\Sigma_n$ and $A_n$ with the understanding that  $\sigma=\sigma_s,\tau=\tau_s$, if the group is $\Sigma_n$, and $\sigma=\sigma_a,\tau=\tau_a$, if the group is $A_n$. All computations have been made using software written for Mathematica 13.1.0~\cite{mathematica}.
\begin{landscape}
\begin{table}[h]
    \centering
    \renewcommand{\arraystretch}{1.6}
    \resizebox{22cm}{!}{
    \begin{tabular}{||c||c||c||}\hline\hline
\textbf{Groups} & \textbf{Weak conjugacy classes in Mod($S_{10}$)} & \textbf{Cyclic factors $[D_{\sigma};D_{\tau}]$}\\
\hline
\hline
\multirow{3}{*}{$A_4$} & $(4,0;[(1 4)(2 3),2;2,2]^{[3]},[(1 2 4),3;3],[(1 3 2),3;3]^{[2]})$ & $[(3,3;(2,3)^{[3]});(3,3;(1,3)^{[3]})]$\\
\cline{2-3} & \makecell{$(4,1;[(1 2)(3 4),2;2,2]^{[3]})$} & $[(3,4;-);(3,4;-)]$\\
\hline
$A_5$ & $(5,0;[(1 5)(2 4),2;2,2],[(2 4)(3 5),2;2,2],[(2 3)(4 5),2;2,2],[(1 2 3 4 5),5;5])$ & $[(3,4;-);(5,2;(1,5),(4,5))]$\\
\hline
$A_6$ & $(6,0;[(1 2)(4 6),2;2,2],[(1 2 4 3)(5 6),4;4,2],[(2 3 4 5 6),5;5])$ & $[(3,4;-);(5,2;(1,5),(4,5))]$\\
\hline
\multirow{2}{*}{$\Sigma_4$} & $(4,0;[(2 3),2;2],[(1 2 4 3),4;4],[(1 2 3 4),4;4]^{[2]})$ & $[(2,5;(1,2)^{[2]});(4,1;(1,4)^{[3]},(3,4)^{[3]})]$\\
\cline{2-3} & \makecell{$(4,0;[(1 4)(2 3),2;2,2],[(2 4),2;2],[(3 4),2;2]^{[2]},[(1 2 3 4),4;4])$} & $[(2,4;(1,2)^{[6]});(4,2;(1,4),(3,4),(1,2)^{[2]})]$ \\
\hline
\hline
    \end{tabular}}
    \caption{The weak conjugacy classes of alternating and symmetric subgroups of $\map(S_{10})$.}
    \vspace{1.3cm}
\resizebox{22cm}{!}{
\begin{tabular}{||c||c||c||}\hline \hline
\textbf{Groups} & \textbf{Weak conjugacy classes in Mod($S_{11}$)} & \textbf{Cyclic factors $[D_{\sigma};D_{\tau}]$}\\
\hline\hline
$A_4$ & $(4,0;[(1 2)(3 4),2;2,2]^{[2]},[(1 2 4),3;3],[(1 3 4),3;3]^{[2]},[(2 3 4),3;3])$ & $[(3,3;(1,3)^{[2]},(2,3)^{[2]});(3,3;(1,3)^{[2]},(2,3)^{[2]}))]$ \\
\hline
$A_5$ & $(5,0;[(1 3)(2 4),2;2,2]^{[2]},[(3 5 4),3;3],[(3 4 5),3;3])$ & $[(3,3;(1,3)^{[2]},(2,3)^{[2]});(5,3;-)]$  \\
\hline
\multirow{2}{*}{$\Sigma_4$} & $(4,0;[(1 3 4),3;3],[(2 3 4),3;3],[(1 2 3 4),4;4]^{[2]})$ & $[(2,6;-);(4,2;(1,4)^{[2]},(3,4)^{[2]})]$\\
\cline{2-3} & \makecell{$(4,0;[(1 2),2;2],[(1 3),2;2],[(1 4)(2 3),2;2,2],[(1 4 3),3;3]^{[2]})$} & $[(2,5;(1,2)^{[4]});(4,3;(1,2)^{[2]})]$  \\
\hline
\multirow{2}{*}{$\Sigma_5$} & $(5,0;[(1 3 2),3;3], [(1 2 5 4),4;4], [(2 3 4 5),4;4])$ & $[(2,6;-);(5,3;-)]$\\
\cline{2-3} & \makecell{$(5,0;[(1 4)(2 3),2;2,2], [(1 3 5)(2 4),6;3,2], [(1 2)(3 4 5),6;3,2])$} & $[(2,5;(1,2)^{[4]});(5,3;-)]$  \\
\hline\hline
    \end{tabular}}
    \caption{The weak conjugacy classes of alternating and symmetric subgroups of $\map(S_{11})$.}
\end{table}
\end{landscape}
\section*{Acknowledgements} The authors would like to thank Dr. Siddhartha Sarkar and Dr. Apeksha Sanghi for helpful discussions.
\bibliographystyle{plain}
\bibliography{main}
\end{document}